%% file: SharkovskyTheorem.tex
\documentclass[12pt]{article}
\usepackage[utf8]{inputenc}

\usepackage{amsmath}
\usepackage{amssymb}
\usepackage{amsthm}
\usepackage{verbatim}
\usepackage{graphicx}
\usepackage{wrapfig}
\usepackage{framed}
\usepackage{caption}
\usepackage{adforn}
\usepackage{subfigure}

\usepackage{colortbl}

\input{macruss}

\usepackage{wrapfig,lipsum}
\usepackage[linewidth=2pt,linecolor=blue, nobreak=false]{mdframed}

\usepackage[colorlinks=true]{hyperref}
 \hypersetup{urlcolor=blue, citecolor=red}

\let\inter=\bigcap

\usepackage {tikz}

\usetikzlibrary {positioning}
\usetikzlibrary{shapes,arrows}
\usetikzlibrary{matrix}
\usetikzlibrary{patterns}

\tikzset{my arrow/.style={
  blue!60!black,
  -latex
  }
}

\usepackage {xcolor}
\definecolor {processblue}{cmyk}{0.96,0,0,0}
\usepackage[margin=2cm]{geometry} 
\usepackage{graphicx}
\usepackage{amsthm}
\usepackage[EULERGREEK]{sansmath}
\usepackage{pifont} 

\newcommand{\R}{\mathbb{R}} 
\newcommand{\N}{\mathbb{N}} 

\theoremstyle{plain} 
\newtheorem{theorem}{Theorem} 
\newtheorem*{theorem*}{Theorem} 
\newtheorem*{prop*}{Proposition} 
\newtheorem{lemma}[theorem]{Lemma} 
\newtheorem*{lemma*}{Lemma} 
\newtheorem{corollary}[theorem]{Corollary}
\newtheorem{example}{Example}  

\begin{document}

\begin{titlepage}
    \begin{center}
        \vspace*{1cm}
        
        \textbf{Sharkovsky Theorem}
        
        \vspace{0.5cm}

        \vspace{1.5cm}
        
        \textbf{Veniamin L. Smirnov, Juan J. Tolosa}
        
        \vfill

        \vspace{0.8cm}
        
        \includegraphics[width=0.4\textwidth]{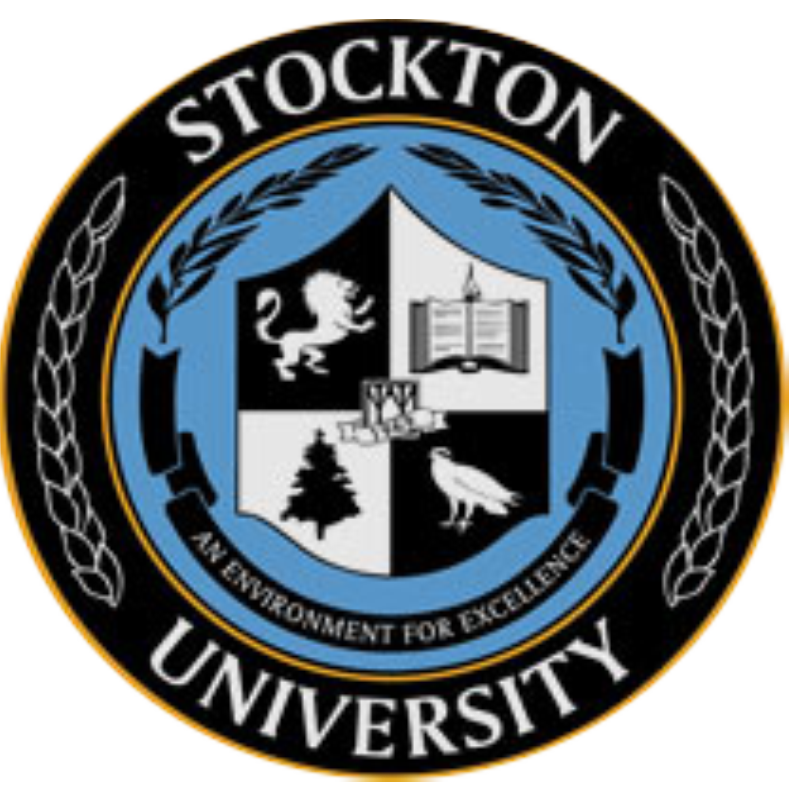}
        
        School of Natural Sciences and Mathematics \\
        Stockton University \\
        The United States of America \\
        June, 2016
        
    \end{center}
\end{titlepage}

\tableofcontents
\newpage

\vspace{12pt}

\begin{abstract}
   We reconstruct the original proof of Sharkovsky Theorem. Specifically, we supply 
   detailed 
    examples and proofs of Lemmata. This version of the proof does not require any 
    sophisticated background from the reader: some experience in Calculus and set theory suffice to comprehend the beauty of the Theorem.
\end{abstract}

\section{Introduction}

Mathematics has a major impact on many branches of human activities. The modern state of science is dedicated to studying complicated processes and phenomena of many kinds. Theories of mathematics are applied in many other sciences, at the first sight, having nothing in common with the Queen of Sciences, such as Law Science
 and Language Studies. 

Mathematics itself in its development has given life to many fields, such as the Theory of Dynamical Systems.

This theory came out from Celestial Mechanics. Its founders were Isaac Newton, Joseph-Louis Lagrange, Pierre-Simon Laplace, William Rowan Hamilton, and Henri Poincar\'e, among others. 

The analysis of the three-body-problem in celestial mechanics, 
stated in 1885, 
led to the birth of  chaotic dynamics and chaos itself, a term that would be developed almost a century later. 
That was the first stage of the modern history of dynamical systems.

The second stage was based on statistical physics and ergodic theory.

Later, with the appearance of micro electronics and  radio electronics, the theory of dynamical systems  entered  the third stage.

In the 20th century there were two main centers of study of  dynamical systems: the USA and the USSR, two countries separated by the Iron Curtain. This led to a parallel development of dynamical systems theory. 

In particular, for the case of one-dimensional dynamics, 
despite the availability of translated papers written by A.N. Sharkovsky on the existence of certain order of natural numbers that explains coexistence of periods, they remained unknown in the USA. 
In 1975 James Yorke and his graduate student Tien-Yien Li,
in a paper called \textit{Period Three Implies Chaos} \cite{Yorke}, proved the following remarkable result: 
if a continuous map of a closed, bounded interval of the line into itself has a period-three orbit, then it also has orbits of every period.  
This, however, was just a particular case of the much more general Sharkovsky Theorem.  
Truth be told, period three is the easiest case, as the reader will notice reading through this paper. 
However, the period-three case already reveals the beauty of one-dimensional dynamical systems.

Let us start by showing a concrete example of a function having a relatively simple orbit of period three; 
we will also clarify the meaning of ``period-three orbit.''

\begin{mdframed}
\begin{example}
\label{period3}

\vspace{12pt}
%
Let us consider the following quadratic function:

\begin{eqnarray*}
f(x)=-\frac{3}{2}x^2+\frac{11}{2}x-2,
\end{eqnarray*}
 whose graph is depicted below.

We say that a value $x$ is a periodic point of period $n$ if the orbit of $x$ under iterations by $f$,
that is, the set $\{x, f(x), f(f(x)), \ldots \}$ has exactly $n$ points.
We also say that this set is a periodic orbit of period $n$, or an $n$-cycle.  
Equivalently, $f^n(x)=x$, but $f^k(x)\ne x$ for $0<x<n$ (note that $f^k$ means $f \circ f \circ \cdots \circ f$, $k$ times). 

{\centering
      \includegraphics[width=10cm,height=7.5cm]{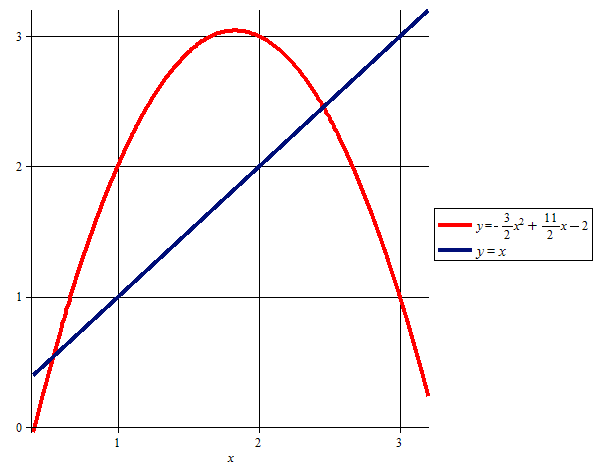}
      \par}
      
This function satisfies
\begin{eqnarray*}
f(1) &=& 2 ,\\
f(2) &=& 3, \\
f(3) &=& 1.
\end{eqnarray*}      
Therefore, the iterations of this function, starting with $1$, produce the following sequence:
\begin{eqnarray*}
1,2,3,1,2,3,1,2, \ldots
\end{eqnarray*}      
This means that 
$x=1$ is a point of period three. Now the motivating question: does the function $f$ have any other periodic points? If so, of what periods?

For very small values of $n$ it may be possible to use algebra to answer this question, 
but as $n$ grows, algebra fails quickly: 
notice that solving $f^k(x)=x$ requires finding the roots of a polynomial of degree $2^k$.

As it turns out,
this function will have cycles {\it of all periods}!

\end{example}     
\end{mdframed}

This rather surprising result will be the topic of Section
\ref{period3section}. 

This theorem amazed the mathematical community. In fact, Li and Yorke's paper was responsible for introducing the word ``chaos'' into the mathematical vocabulary. (Now there are several competing definitions of chaos.)

One-dimensional dynamical systems, 
induced by a continuous map of an interval $I$ into itself, 
have an important property: the points of a particular trajectory provide plenty  information about 
the dynamical system itself.  
This is a characteristic of the one-dimensional  space---an interval $I$. 

The existence of certain periods leads to coexistence of other periods.  This feature was observed by the Ukrainian mathematician Aleksandr Nikolaevich Sharkovsky.

In 1964, he introduced  the following ordering on the positive integers,which we will denote as  $n \triangleright m$ \cite{Shark1}:

\begin{eqnarray*}
&3 \triangleright 5 \triangleright 7 \triangleright 9 \triangleright \ldots \\
&2 \cdot 3 \triangleright 2 \cdot 5 \triangleright 2 \cdot 7 \triangleright 2 \cdot 9 \triangleright \ldots \\
&2^2 \cdot 3 \triangleright 2^2 \cdot 5 \triangleright 2^2 \cdot 7 \triangleright 2^2 \cdot 9 \triangleright \ldots \\
&\ldots\\
&\ldots 2^5 \triangleright 2^4 \triangleright 2^3 \triangleright 2^2 \triangleright 2 \triangleright 1 .
\end{eqnarray*}

First come the odd numbers (except one), 
then the doubles of the odd numbers, then $2^2$ times each odd number, etc. When all of these values are exhausted, the ordering ends with the decreasing powers of $2$ (including $2^0 = 1$).

Observe that the Sharkovsky ordering includes all natural 
numbers ordered in a special way. 
Indeed, 
any positive integer that is not a power of $2$ can be written as $2^l \cdot m$, where $m$ is an odd number,   
$m>1$ , and $l \geq 0$ (\cite{Don}, p.~47).

Sharkovsky Theorem states 
that if a continuous function from $I$ to $I$ has
a cycle of a given period $n$, then it also must have a cycle of period $m$,
for every $m$ satisfying $n \triangleright m$ in the above ordering.

Since $3$ is the first number of the Sharkovsky ordering, by this Theorem, a cycle of period $3$ implies the existence of cycles of periods of all integers after $3$ in the order above. 
Hence, ``period three implies all periods" follows directly from this Theorem. 
However, besides $3$, we get information of what happens 
for all other natural numbers!

This Theorem helps us understand ``the world" of one-dimension dynamical systems. 

Besides the original proof, done by A.N. Sharkovsky, there exist several other proofs,
which have been completed using  various sets of ``mathematical tools" \cite{BGMY}, \cite{BurnsHass}, \cite{48}, \cite{75}, \cite{84}, \cite{104}, \cite{110}, \cite{113}.

However, the original proof is distinct from the others by its structure. 
This proof not only reflects the path which was followed by Sharkovsky in order to come to his conclusions,
but also contains multiple auxiliary results that have interest on their own.

All publications of the proof of Sharkovsky Theorem (including the original one, which is not widely accessible, is very compact, and has some typographical errors) have targeted experienced mathematicians with sufficient background on Dynamical Systems and neighboring branches of mathematics. 

Our aim is to provide a proof that even students with some experience in Calculus I and set theory with a grain of interest could follow, and thus understand the beauty of the Theorem. 
For this reason, we have striven to illustrate all the concepts with multiple examples, pictures, and graphs.

\vspace{12pt}

\bigskip

\section{A brief biography}
\label{history}

Aleksandr Nikolaevich Sharkovsky was born on 7 December 1936 in Kiev. 

In 1953 he was accepted as a student of the Physico-mathematical School of the Kiev University named after T.G. Shevchenko. During the first year Aleksandr Nikolaevich attended a Higher Algebra coterie hosted by 
professor G. E. Shilov, 
where A. N. Sharkovsky delivered a series of talks  
on iterating processes. 
One of his topics was ``The conditions of convergence of the iterating sequences $f^n(x)$ to a fixed point.''

\begin{wrapfigure}{R}{0.3\textwidth}
  \begin{center}
    \includegraphics[width=0.25\textwidth]{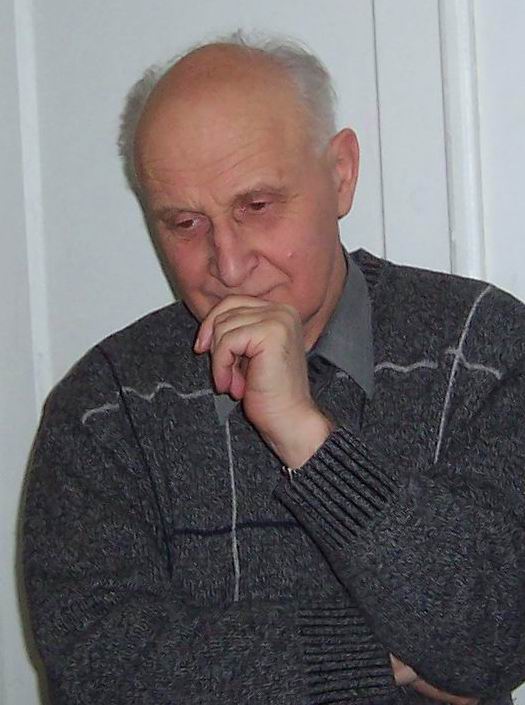}
  \end{center}
  \caption{A.N. Sharkovsky}
\end{wrapfigure}

During a 2-year course on Real Analysis, the analysis of so-called bad functions was presented, including their sequences and sums. The idea of considering only static conditions of bad functions without considering their dynamics left him with a feeling of dissatisfaction. The area of interest was chosen.

In the year 1958,in order to graduate, Aleksandr Nikolaevich had to write his Senior Thesis, 
an integral part of the educational process in the Soviet Union.  
Being a good student, he was asked to work on either the theory of probability, the theory of functions of complex variables, or functional analysis. All his attempts failed.

During his 5th year in the University, he did internship working with ECM (Electronic Computing Machines). His duties included writing and developing programs for the ECM. He came out with some ideas on how to improve some processes performed by ECM. Later, when the deadline for the senior thesis got closer, he decided to implement his experience with ECM in a paper 
and chose the address programming as the topic for the thesis.  However, this area of research was not approved by the faculty. And since time to make a decision was almost over, Aleksandr Nikolaevich decided to study iterating processes. 
He also got a feeling that he could come up with some interesting ideas. The topic for the senior thesis was picked.  

His thesis was said to be distinct. 

Naturally, the desire to continue with iterations followed Aleksandr Nikolaevich during his post-graduate studies. However, this idea was not widely spread among mathematicians and it led to 
problems with a thesis advisor. 
Finally, 
N. N. Bogolyubov, who lived in Moscow and studied similar areas of mathematics, became his advisor. 
However, because of the distance, Aleksandr Nikolaevich never met him. All discussions about the thesis happened to be with Yu.~A. Mitropolsky, whose name appeared on the thesis as advisor. 

Half a year before graduating from Post Graduate School in 1961, the thesis ``Some questions of the theory of one-dimensional iterating processes'' was presented. He showed that if a continuous map has a cycle of period $2^m$, then it has cycles of periods $2^i, i=0,1,2,\ldots, m-1$. 
Therefore, back on those days he had already found a piece of his order:

$$\cdots 2^5 \triangleright 2^4 \triangleright 2^3 \triangleright 2^2 \triangleright 2 \triangleright 1.$$

 That paper became a foundation for the next thesis and the Theorem itself. 
 
 The work on iterating processes went on.
 
 In March 1962 Aleksandr Nikolaevich published a paper called ``Coexistence of Cycles of a Continuous Mapping of the Line into Itself" that contained the Theorem and complete proof \cite{Shark3}, \cite{Shark1}.
 
 Nowadays, Aleksandr Nikolaevich works for The Institute of Mathematics of NAS of Ukraine%
 \footnote{The Institute of Mathematics of National Academy of Science of Ukraine. 01601 Ukraine, Kiev-4, Tereschenkivska str., 3,  \url{http://www.imath.kiev.ua/}.}
 (since July 1, 1961); he is the head of the department of the Theory of Dynamical Systems.

 \section{Preliminary Definitions and Notation}
 \label{defin}

Let $f$ be a continuous map
$f \colon I \to \mathbb R$, with $I \subset \mathbb R$. 
Given sets $A \subseteq I$ and $B \subseteq \mathbb R$, the notation
$A \to B$ will mean ``A covers B'' (under $f$), that is, $f(A) \supset B$.
\footnote{The symbols $\subset$ and $\supset$ will be used to denote
non-strict inclusions, what is usually denoted by $\subseteq$ and $\supseteq$
respectively.}
The given continuous function $f$ should be clear from the context.

We denote iterations of a function with itself by a superscript, so that
\begin{eqnarray*}
f^0(x_0) &=& x_0 ,\\
f^1(x_0) &=& f(x_0) ,\\
f^2(x_0) &=& f(f(x_0)) ,\\
\ldots  \\
f^m(x_0) &=& \underbrace{f (f\cdots(f}_{\text{m-times}}(x_0))\cdots).
\end{eqnarray*}
The set of interior points of an interval $I$ is denoted by  Int$(I)$.

A point $x_0$ is said to be of {\it period $m$}, if 
\begin{eqnarray*}
f^m(x_0)=x_0.
\end{eqnarray*}
The least positive $m$ for which $f^m(x_0)=x_0$ is called the 
{\it prime period}, or {\it the least period\/} of $x_0$.
In particular, we can say that a fixed point is a point of period $1$.

A {\it cycle of period $m$\/} 
is the set of points consisting of a point $x_0$ of period $m$ together with its iterates, that is, 
\begin{eqnarray*}
\{x_0, x_1, \ldots, x_{m-1}\},
\end{eqnarray*} 
where $x_i=f(x_{i-1})$ for 
$i=1,2,\ldots m-1$  
and $f(x_{m-1})=x_0$.

\section{Two fixed-point theorems}
\label{fixedPts}

\bigskip

In this section, we will state and prove two important fixed-point theorems, which will be widely used in our work. 
These theorems have also independent interest.

\begin{theorem}
\label{fixed1}
Let $I$ be a closed bounded interval $[a,b] \subset \mathbb R$, 
and let $f: I \to I$ be continuous on $I$. 
Then there exists a fixed point $\gamma \in I$ of $f$, that is, $f(\gamma)=\gamma$.
\end{theorem}

\begin{proof}
Let $g(x)=f(x)-x$, which is continuous, as a difference of two continuous functions.

\begin{figure}[ht!]
\centering
\includegraphics[width=70mm]{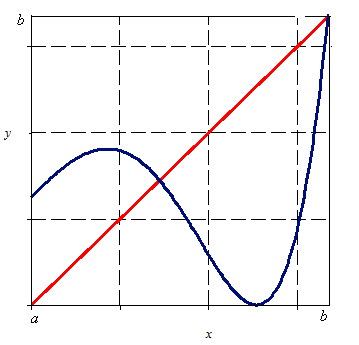}
\caption{}
\label{fixed_int}
\end{figure}
Since   $f: I \to I$ (see Figure \ref{fixed_int}), we get the following possibilities at
each endpoint:

$$f(a) = a, \quad\hbox{ or}\quad f(a)>a;  $$
$$f(b) = b, \quad\hbox{ or}\quad f(b)<b . $$
If either $f(a)=a$ or $f(b)=b$, then we have a fixed point and the proposition follows. 

If $f(a)>a$ and $f(b)<b$, then 
$$g(a)=f(a)-a>0, $$
$$g(b)=f(b)-b <0.$$
Since $g$ is a continuous function on $[a,b]$, 
by the Intermediate Value Theorem (Bolzano's theorem), 
there must exist a point $\gamma \in I$, such that $g(\gamma)=0$, which implies

$$f(\gamma)=\gamma.$$

This completes the proof.

\end{proof}

\medskip

The following is our second fixed-point result. 
Here, again, $I$ is a closed bounded interval of $\Bbb R$.

\begin{theorem}
{\label{fixed2}
Let $f: I \to \mathbb R$ be a continuous function, and assume that
$ I \subset f(I)$, so that $I$ covers itself under $f$. 
Then  there exists at least one fixed point $\gamma \in I$ of $f$,
that is,  $f(\gamma)=\gamma$.}
\end{theorem}

\begin{proof}

Let $I$ be the closed interval $[a,b]$.
Since $f$ is continuous, by Weierstrass's theorem (the Extreme Value Theorem) 
and the Intermediate Value theorem,
we know that $f(I)$ is another closed and bounded interval. Let us denote it by 
$f(I) =[\alpha, \beta]$.
Since $f$ attains its maximum and minimum values on $I$, 
 then there exist at least two points $a', b' \in I$, so that 
(see Figure \ref{fixTwo})
$$f(a')= \min_{[a,b]} f=\alpha \leq a \leq a' \leq b \leq \beta,$$
$$f(b')=\max_{[a,b]} f= \beta \geq b \geq b' \geq a \geq \alpha.$$
\begin{figure}[ht!]
\centering
\includegraphics[width=90mm]{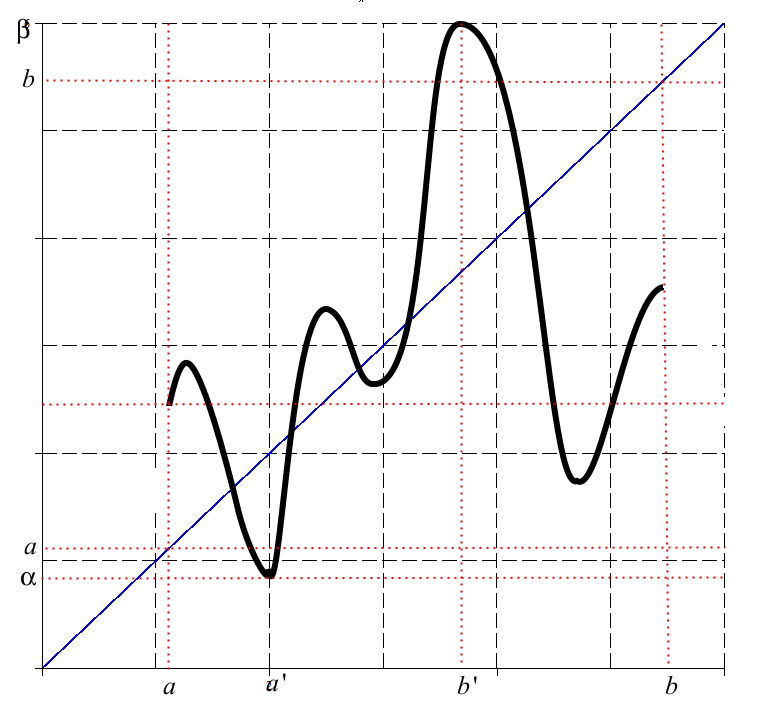}
\caption{}
\label{fixTwo}
\end{figure}

As in the proof of Theorem \ref{fixed1}, one can find two situations at either $a'$ or $b'$:

\begin{eqnarray*}
\alpha = f(a')=a' \quad \text{or}  \quad f(a')<a' ;  \\
\beta = f(b')=b' \; \quad \text{or}  \quad  f(b')>b'. \\
\end{eqnarray*}

If either $f(a')=a'$ or $f(b')=b'$, we can pick $\gamma=a'$ or $\gamma=b'$ and the proof is complete.

Consider the second possibility:

$$f(a')< a' \qquad\hbox{ and}\qquad f(b') > b'.$$

Let $g$ be the continuous function defined as $g(x)=f(x)-x$, so that

$$g(a')=f(a')-a' <0,$$
$$g(b')=f(b')-b' >0.$$
Therefore, 
$g$ changes its sign on the interval with endpoints $a',b'$
By the Intermediate Value Theorem, there must exist a point 
$\gamma$ in that interval, at which $g(\gamma)=0$. 
Since this interval is contained in $[a, b]$, we conclude that
there is a $\gamma \in [a,b]$, such that
$$g(\gamma)=0 , $$
which implies

$$f(\gamma)=\gamma.$$

This result completes the proof.
\end{proof}

\smallskip
{\bf Note.} While Theorem \ref{fixed1} is also valid in higher dimensions 
(for a map of the $n$-dimensional cube into itself, this is the famous Brouwer fixed-point theorem),
it is interesting to observe that Theorem \ref{fixed2} already fails in dimension two. 
This is shown by a nice example by Nik Weaver 
(\url{http://mathoverflow.net/questions/211459}).

\section{Itinerary Lemma}
In this section, we study the connection between 
a sequence of intervals covering each other, and forming a closed graph, 
and the existence  of  periodic points.

Given a continuous function $f \colon \Bbb R \to \Bbb R$, and sets $A, B \subset \Bbb R$,
we defined $A \rightarrow B$ in Section \ref{defin},  
to mean that 
$A$ covers $B$ under  $f$, that is, 
$f(A) \supset B$.

Further, if $B \rightarrow C$, then, similarly, $f(B) \supset C$.  

All together, 
we write the sequence $A \rightarrow B \rightarrow C$.

The Itinerary Lemma will show the connection between periodic points and a loops of covering intervals.

The following result, which also has independent interest, 
will be used in the proof of the itinerary lemma.
Let $I$ and $J$ be closed, bounded intervals of $\Bbb R$. 

\begin{lemma}
\label{it_lemma}

If $f$ is continuous on $I$ and $I \to J$, then there exists an interval $K$ such that $K \subset I$ and $f(K)=J$.

\end{lemma}

\begin{proof}

Let us denote $I = [a, b]$ and $J = [c, d]$. Let us call

$$ A = f^{-1}(c) \cap I 
\qquad\hbox{and}\qquad B = f^{-1}(d) \cap I . $$
Since $f$ is continuous, and $\{c\}, \{d\}$ are closed sets in $\mathbb R$,
then both $A$ and $B$ are closed. Since we assume that $f(I) \supset J$,
then both $A$ and $B$ are nonempty,  bounded, 
disjoint subsets of $I$. Therefore, both $A$ and 
$B$ have a maximum and a minimum element in $I$.

We will distinguish two cases.

{\sc Case 1.} $\max B < \min A$ (See Figure \ref{picB}).

\begin{figure}[ht!]
\centering
\includegraphics[width=80mm]{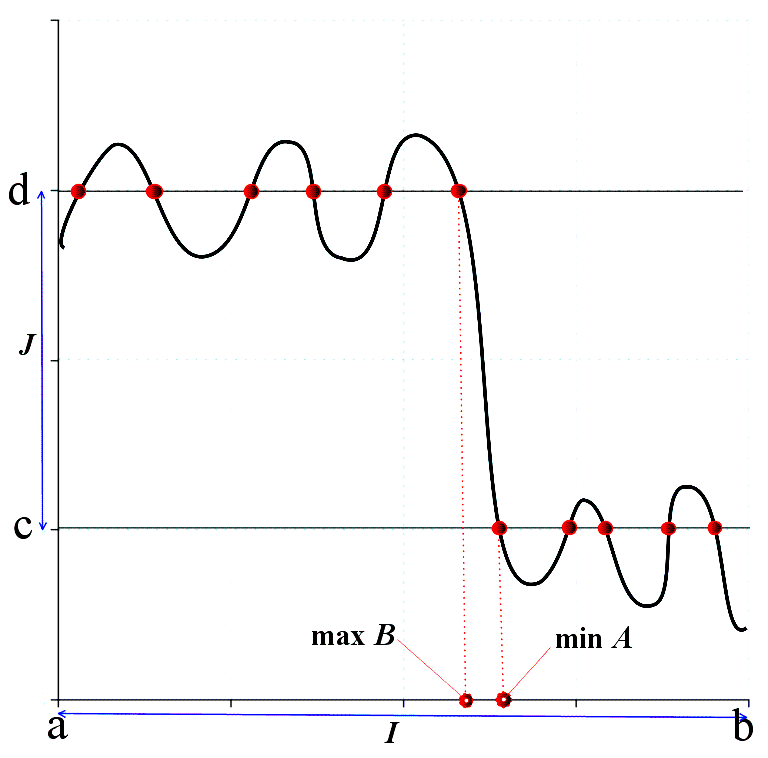}
\caption{First case}
\label{picB}
\end{figure}

Call $\alpha = \max B$, $\beta = \min A$, and $K = [\alpha, \beta]$.
We have $K \subset I$; we claim that $f(K) = J$. Indeed, by the Intermediate Value Theorem,
$f(K) \supset J$.

Assume, by contradiction, that $f(K)$ is strictly larger than $J$. Then there exists an 
$y\in f(K) \setminus J$. Then either $y > d$, or $y < c$.

If $y > d$, then there exists an $x \in [\alpha, \beta]$ such that $f(x) = y > d$.
Hence, by the Intermediate Value Theorem, there exists an $s \in [x, \beta]$ such that 
$f(s) = d$; hence, $s \in B$ (Figure \ref{picC}). 
\begin{figure}[ht!]
\centering
\includegraphics[width=80mm]{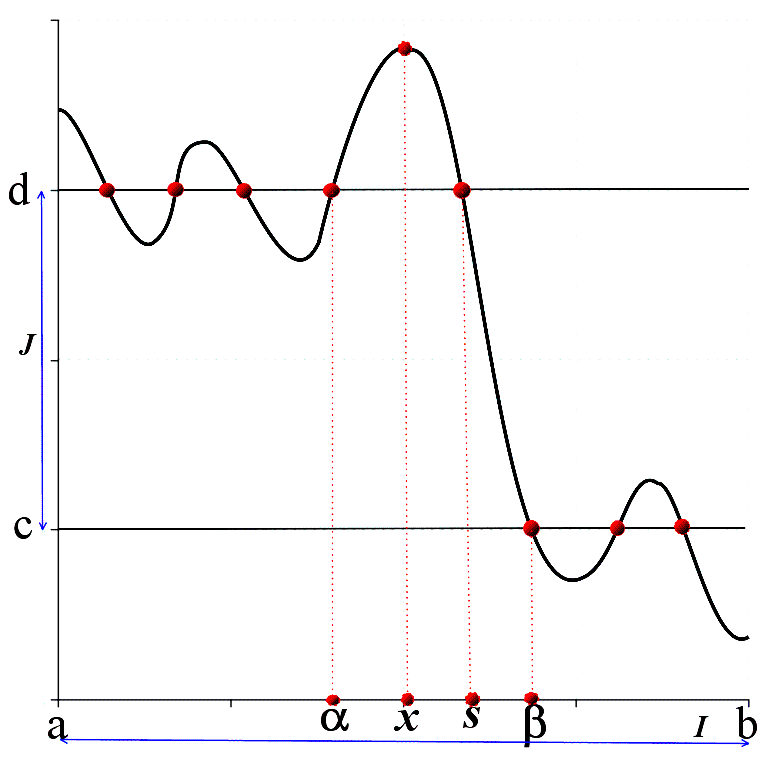}
\caption{}
\label{picC}
\end{figure}
Since $\alpha < x < s$, this contradicts the fact that $\alpha$ was the
largest element in $B$.

The case $y < c$ is ruled out similarly, by obtaining a point $s \in (\alpha, \beta)$ for
which $f(s) = c$, contradicting the fact that $\beta = \min A$. 

\smallskip

{\sc Case 2.} $\min A < \max B$. (See Figure \ref{picD}).

\begin{figure}[ht!]
\centering
\includegraphics[width=80mm]{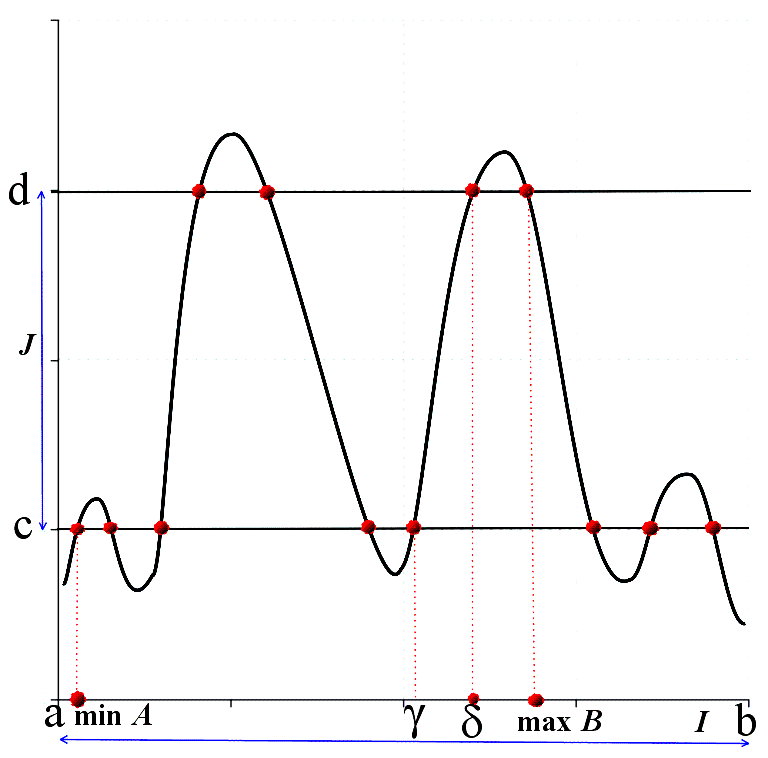}
\caption{Second case}
\label{picD}
\end{figure}

The set $C = A \inter [a, \max B]$ of those elements in $A$ that are less than $\max B$ is 
then closed and nonempty. Let us call $\gamma = \max C$.

Next,  the set $D = B \inter [\gamma, \max B] = B \inter [\gamma, b]$ 
of those elements in $B$ that 
are between $\gamma$ and $\max B$, is also closed and nonempty;
let us denote $\delta = \min D$.

Finally, we define $K$ as the interval $[\gamma, \delta]$. 
As in Case 1, using the Intermediate Value Theorem one shows that $f(K) \supset J$.
To show that actually $f(K)$ equals $J$, one argues by contradiction, as in Case 1, 
and one obtains a contradiction either with the definition of $\gamma$, or the definition of 
$\delta$. 
\end{proof}


\begin{lemma}[Itinerary Lemma, \cite{BurnsHass}, \cite{Shark2}]
\label{Itin1}  
1) If $I_0, I_1, I_2, \cdots, I_{n-1}$ are closed bounded intervals and  $I_0 \rightarrow  I_1 \rightarrow I_2 \rightarrow \cdots \rightarrow I_{n-1} \rightarrow I_0$, then there is a fixed point $\gamma$ of $f^n$ such that 

$$f^i(\gamma) \in I_i$$
for $0 \leq i <n$, and 

$$f^n(\gamma)=\gamma.$$

We say that such $\gamma$ {\it follows the given loop of intervals}.

2) Moreover, if the endpoints of all the intervals $I_i$ belong to a given $k$-orbit 
$B = \{\beta_1, \beta_2, \ldots, \beta_k\}$ which does not follow the loop,%
{\color{blue}
 \footnote{Meaning that no point of $B$ follows the loop.}
 }
and if 
the interior 
$Int(I_0)$ of $I_0$ is disjoint from every $I_1, \ldots I_{n-11}$, 
then $\gamma$ has least period $n$.

\end{lemma}

\begin{proof}

1) Assume that 

$$f(I_i) \supset I_{i+1}$$
for every $i=0 \cdots n-2$, and
$$f(I_{n-1}) \supset I_0.$$

By Lemma \ref{it_lemma}, there exists a closed bounded interval  $K_{n-1}$ 
such that $K_{n-1} \subset I_{n-1}$ and
$$f(K_{n-1})=I_0.$$
Further, let $K_{n-2}$ be a closed interval such that
\begin{eqnarray*}
K_{n-2} \subset I_{n-2} \qquad \hbox{and} \qquad f(K_{n-2})=K_{n-1} \subset I_{n-1}. 
\end{eqnarray*}

Therefore,
\begin{eqnarray*}
f^2(K_{n-2})=f(K_{n-1})=I_0. \\
\end{eqnarray*}

By induction, let $K_{n-j}$ be a subinterval of $I_{n-j}$ such that $f(K_{n-j})=K_{n-j+1}$, 
for $j=1,2,3,\ldots n $. 
In particular we get $K_0 \subset I_0$, so that

$$f(K_0)=K_1 \subset I_1, $$
$$f^2(K_0)=f(K_1)=K_2 \subset I_2, $$
$$f^3(K_0)=f(K_2)=K_3 \subset I_3, $$
$$\cdots                 $$
$$f^n(K_0)=f(K_{n-1})=K_n \subset I_0.$$

Since, $K_0 \subset I_0=f^n(K_0)$, by 
Theorem \ref{fixed1} of Section \ref{fixedPts}, 
we get a fixed point $\gamma \in K_0$ of $f^n$ (i.e. $f^n(\gamma)=\gamma)$.   Moreover, $f^i(\gamma) \in f^i(K_0)=K_i \subset I_i$ for $i=1,2,3,\ldots n$, so that $\gamma$ indeed follows the loop.

2) Under the extra assumptions, if $\gamma$ follows the loop, then $\gamma \not\in B$, hence
$\gamma$ is an interior point of $I_0$. If $k$ is any positive integer less than $n$, then $f^k(\gamma)$
belongs to $I_k$, which is disjoint from ${\rm Int}(J_0)$. Hence, $f^k(\gamma) \neq \gamma$ for 
every such $k$, which proves that $\gamma$ has least period $n$. 
\end{proof}

\section{Period three implies all periods}
\label{period3section}

The result in this section is a particular case of Sharkovsky theorem. 
However, we believe this case is important on its own and, moreover, 
it illustrates some of the ideas and techniques used for the general theorem. 
Moreover, for historical reasons this case was the first to appear in Western literature.

\bigskip

\begin{theorem}
\label{chaos}
Let $f$ be a continuous function of a closed bounded interval $J$ of $\mathbb R$ into
itself. 
If $f$ has an orbit or period three, then it has orbits of all periods.%
\footnote{Actually, as the proof shows, the assumption that $f$ sends a closed bounded interval
                into itself is not required for this particular result.}
\end{theorem}

\begin{proof}
To fix ideas, 
let us 
consider the following example of a map from the quadratic family 
\begin{eqnarray*}
f(x)= \lambda x(1-x) 
\end{eqnarray*}
for the particular choice of 
$\lambda =1+\sqrt{8}$ of the parameter, so that
\begin{eqnarray*}
f(x)= (1+\sqrt{8}) x (1-x) .
\end{eqnarray*}
One can check that $f$ sends the unit interval $I = [0, 1]$ into itself. 
Moreover, one can show that 
the map exhibits a cycle of period three, namely, 
\begin{eqnarray*}
f(\beta_1) &=& \beta_2, \\
f(\beta_2) &=& \beta_3, \\
f(\beta_3) &=& \beta_1, \text{\;\; or\;\;} \\
f^3(\beta_1) &=& \beta_1, 
\end{eqnarray*}
where $\beta_1 \approx 0.15992689$.

\begin{figure}[ht!]
\centering
\includegraphics[width=80mm]{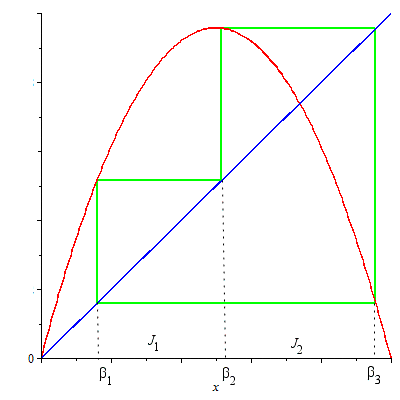}
\caption{Cycle of period 3, $\lambda=1+\sqrt{8}$}
\label{3cycle}
\end{figure}

As shown in Fig.~\ref{3cycle}, let $J_1$ be the interval with endpoints $\beta_1, \beta_2$, so that $J_1=[\beta_1, \beta_2]$; similarly, let $J_2=[\beta_2, \beta_3]$.%
\footnote{
The arguments that follow will apply to {\it any\/} map of an interval into itself 
which has a three-orbit, with a suitable labeling of the intervals $J_1$ and $J_2$.}

A similar behavior is exhibited by the piecewise-linear map shown in Figure \ref{piecewise}. 

Observe that $f: I \to I$ with $J_1, J_2 \subset I$, and $f$ is continuous.
Moreover, 
$J_1 \to J_2$ and $J_2 \to J_1 \cup J_2$, which means that $J_1$ covers $J_2$, 
and $J_2$ covers itself and $J_1$.
This can be illustrated with the following graph:

\begin{figure}[ht!]
\centering
\includegraphics[width=80mm]{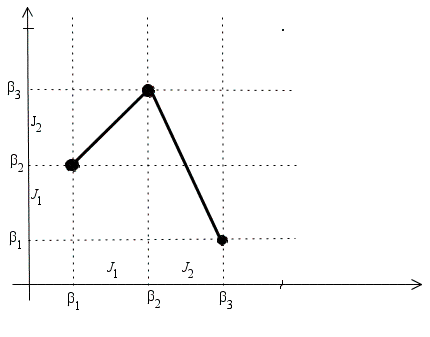}
\caption{Cycle of period 3}
\label{piecewise}
\end{figure}

\begin {center}
\begin {tikzpicture}[-latex ,auto ,node distance =2 cm and 3cm ,on grid ,
semithick ,
state/.style ={ circle ,top color =white , bottom color = processblue!20 ,
draw,processblue , text=blue , minimum width =0.4 cm}]
\node[state] (A)
{$J_1$};
\node[state] (B) [right=of A] {$J_2$};
\path (A) edge [bend right = -15] node[right] {} (B);
\path (B) edge [bend left =15] node[below =0.15 cm] {} (A);
\path (B) edge [loop right =25] node[above] {} (B);
\end{tikzpicture}
\end{center}

Since $J_2 \to  J_2$, by Theorem \ref{fixed2} we get a fixed point, that is,  a point of period 1.  

Next, 
note that there is a loop $J_1 \to J_2 \to J_1$; 
then by the Itinerary Lemma (Lemma \ref{Itin1}) there exists a point $\mu \in J_1$ that follows the loop, such that $f^2 \mu = \mu$; this will be a point of period two. 
(It is clear this point must belong to the interior of $J_1$, since the endpoints of $J_1$ are part 
of an orbit of period three.)

Thus, period 3 implies both period 1 and period 2.

Let $m$ be an integer, with $m > 3$,
and
consider the following loop:

\begin{equation}
\label{mLoop}
J_1 \to \underbrace{J_2 \to J_2 \to J_2 \to \ldots \to J_2}_{m-1 \text{ copies of } J_2} \to J_1.
\end{equation}

By the Itinerary Lemma (Lemma \ref{Itin1}) 
we conclude that there is $\xi \in J_1$ 
which follows this loop, such that  $f^m(\xi)=\xi$. 

Let us now prove that $m$ is the least period of $\xi$.
Indeed, first of all, Int$(J_1)$ is disjoint from $J_2$. 
Next, the endpoints of $J_1$ and $J_2$ are elements of the
three-orbit $B = \{\beta_1, \beta_2, \beta_3\}$, which does {\it not\/}
follow the loop (\ref{mLoop}).
Indeed, the first two terms in this loop can only happen for the point $p = \beta_2$
from $B$. 
But then $f(p) = \beta_3$, and $f^2(p) = \beta_1 \notin J_2$. Since $m-1$ is
at least two, we see that $B$ does not follow the loop (\ref{mLoop}).  
By the second part of Lemma \ref{Itin1}, we conclude that $m$ is indeed the least period for
$\xi$.

He have thus shown that this map
(or, in general, any continuous map of an interval into itself, which
has a three-orbit) has a periodic orbit of prime period $m$, for any positive integer $m$.
\end{proof}

\section{Important Lemmata}
\label{lemmata}

In preparation for Sharkovsky theorem, we will prove several results, which have 
independent interest.

Here we are going to use the following notation: if $a,b \in \R$, then 

$\langle a, b \rangle$

is a closed interval with endpoints $a,b$;  it is either unknown or irrelevant which one is greater, $a$ or $b$.

\begin{lemma}
\label{lemma1} 
A continuous map $f: I \to I$ has a cycle of period $2$ if and only if there exists a point $\xi \in I$, such that $\xi \neq f(\xi)$ and 
$ \xi \in f(\langle \xi, f(\xi) \rangle)$.
\end{lemma}

\begin{proof}

Assume first that $f$ has a point of period two.
Recall that a point $a$ has  prime period $m$, if $f^m(a)=a$, and $m$ is the smallest positive integer for which this holds.
In particular, for a point of period 2, one gets
\begin{eqnarray*}
f^2(a)=f(f(a))=a.
\end{eqnarray*}
Then the following is true:

$$a \neq f(a) \qquad\hbox{ and}\qquad a \in 
f(\langle a, f(a)\rangle),$$

which proves the forward direction.

For the backward direction, 
we assume that there exists a point $a \in I$ such that $a \neq f(a)$ and 
$a \in f(\langle a, f(a)\rangle).$

By trichotomy, we have two possibilities:
\begin{eqnarray*}
(i) \;\; f(a)>a, \\
(ii) \;\; f(a)<a.
\end{eqnarray*}
Consider the first case (see Figure \ref{case1}).

\begin{figure}[ht!]
\centering
\includegraphics[width=80mm]{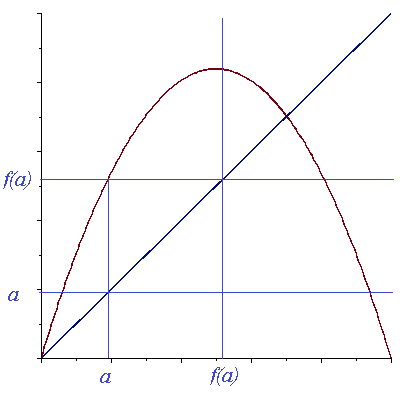}
\caption{Case (i): $f(a)>a$.}
\label{case1}
\end{figure}

Here we have the following assumptions:
$$a \neq f(a), \qquad\hbox{ }\qquad 
a \in f(\langle a, f(a)\rangle), 
\qquad\hbox{ and}\qquad f(a)>a.$$
By the provided assumption 
$a \in f(\langle a, f(a)\rangle)$, 
and by definition of the image of a set under a function, 
there exists a point $c$ on the interval $[a, f(a)]$ 
such that $f(c)=a$. 
Note that $c$ cannot be equal to $a$, since $f(a) = a$ is not true under our assumptions.
Therefore, $c$ lies on $(a, f(a)]$.

If $c$ is  the rightmost endpoint, we have $f(c)=f(f(a))=a$ and since $f(a) \neq a$, we get a point of period 2.

If not, then $c \in  (a, f(a))$, whence $c<f(a)$.

Let us consider the set of points of $I$ bigger than $c$. 

Proceed by cases: 1) there are fixed points bigger than $c$ in $I$;
and 2) there are no fixed points bigger than $c$ in $I$.

\vspace{12pt}

\textbf{1) There are fixed points bigger than $c$.}
 
Since $f(x)$ is continuous, then so is $g(x) = f(x) - x$. Since fixed points of $f$ are zeros of $g$, and since the set of zeros of $g$ is $g^{-1}(\{0\})$, which is the inverse image of a closed bounded set under a continuous function, then this set is closed and bounded (as a subset of $I$). 
Therefore, this set has a smallest element. 

Denote such smallest fixed point as $\gamma$, so that $\gamma =\min \{x=f(x)\Big \vert x\geq c\}$;
in particular, we have $f(\gamma)=\gamma$.

\begin{figure}[ht!]
\centering
\includegraphics[width=90mm]{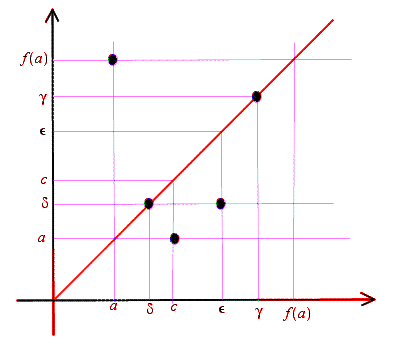}
\caption{Image for the case when $f(a)>a$.}
\label{case11}
\end{figure}

Since we assume that $f(a)>a$ and $f(c)=a$, one can notice that $f([a,c]) \supset [a,c]$. 
By Theorem \ref{fixed2}, there is at least one fixed point $\delta$ of $f$ in $[a, c]$,
so that $f(\delta) = \delta$ (Figure \ref{case11}).

Consider the interval $[c, \gamma]$. 
Since $a < c < \gamma$, we conclude that
$f([c, \gamma]) \supset [a,c]$.

By the Intermediate Value Theorem, since 
$a < \delta < c$,
there must exist a point $\epsilon \in [c, \gamma]$, such that $f(\epsilon)=\delta$.

Consider the dynamics of the point $c$. We know that $f(c)=a$, hence $f(f(c))=f(a)$. 
Keeping in mind that we are considering the case when $c <f(a)$, one can see that $f(f(c))=f^2(c)>c$.

Now, we need to see the dynamics of $\epsilon$. To be precise, $f(\epsilon)=\delta$, but $\delta$ is a fixed point, therefore
$f^2(\epsilon) = f(f(\epsilon)) = f(\delta) = \delta$. 
Since $a<\delta<c$, we can deduce, that $f^2(\epsilon)<c$.

Therefore, there exists a point $x \in [c, \epsilon]$ with $f^2(x)=x$. 
Note that $x$ has true period 2. Indeed, since $c\leq x \leq \epsilon < \gamma$,  and $\gamma$ is the least fixed point, therefore $x$ cannot be a fixed  point.

\vspace{12pt}

Let us now consider the second case.

\textbf{2) There are no fixed points bigger than $c$ in $I$.}
Remember that we are assuming that $f \colon I \to I$, where $I$ is a closed bounded interval.

We claim that there is a point $d > c$ such that $f^2(d) \leq d$. 
Indeed, if we had $f^2(x) > x$ for every $x > c$, then in particular this would hold for the rightmost point of
the interval $I$, leading to a contradiction with the assumption that $f$ maps $I$ into itself and, therefore,
so does $f^2$. 

Note that $f^2(c) > c$. Indeed, $f(c) = a$, so that $f^2(c) = f(a)$, and recall that $c \in (a, f(a))$, so that $f^2(c) = f(a) > c$.

Let us introduce the new function $g(x)=f^2(x)-x$. As a combination of continuous functions, $g(x)$ itself is  continuous. At the points $c,d$ we have
\begin{eqnarray*}
g(d)\leq 0  \qquad\hbox{ and } \qquad
g(c)> 0.
\end{eqnarray*}
Therefore, by the Intermediate Value theorem, there exists a point $x \in [c,d]$, such that $g(x)=0$, which implies $f^2(x)=x$. But since $x > c$, and we assumed there are no fixed points bigger than $c$ on $I$, therefore, the point $x$ is a point of period 2.

This completes the case when $f(a) \neq a$, $a \in f((a, f(a))$.

For the case when $f(a) < a$ and $a \in f((f(a), a))$, assume that $I=[p,q]$, where $p,q$ are the end points of the interval $I$. 
Let us introduce the continuous function $g(x)=-f(x)$. 
Let us also call $J = [-q, -p]$, so that $g \colon J \to J$. 

Then, assuming the second case holds for $f$, that is, $f(a) < a$, and calling $a' = -a$, 
we can see that $g(a') > a'$, so that the first case holds for $g$. Therefore, 
by the first case already proved, 
$g$
has a two-cycle $\{t, g(t)\}$, and this implies that $f$ has the two-cycle
$\{-t, f(-t)\}$.

This completes the proof of Lemma \ref{lemma1}.
\end{proof}

The scheme of the proof of this result is in Figure \ref{sketch_proof}.

\begin{figure}[ht!]
\centering
\includegraphics[width=130mm]{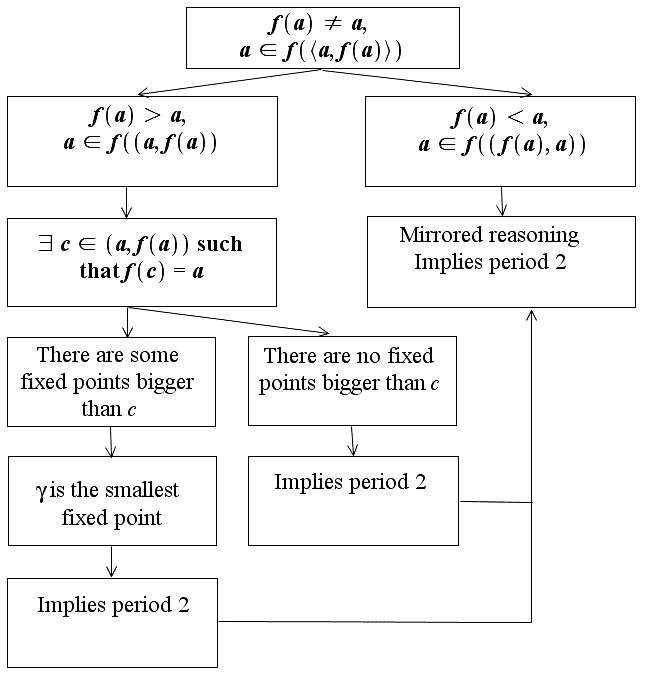}
\caption{The scheme of the proof by cases.}
\label{sketch_proof}
\end{figure}

\begin{lemma}
\label{lemma2} 
If $f$ has a cycle of 
period $m$, with $m >2$, then $f$ has a cycle of period 2.

\end{lemma}

\begin{proof}

Let $B$ be a cycle of period $m$ of a map $f: I \to I$, 
and let 
$\beta_0 =\max\{\beta \in B \Big \vert f(\beta) > \beta\}$.

 We have two different cases:
 \vspace{12pt}
 
\textit{(i)} The interval $[\beta_0, f(\beta_0)]$ does not contain any points of cycle $B$.

\textit{(ii)} The interval $[\beta_0, f(\beta_0)]$ has some points of cycle $B$.

\vspace{12pt}

Consider the first case. Let the first iterations of $\beta_0$ be $f(\beta_0)=\beta_1$ and $f(\beta_1)=\beta_2$. 

Since $\beta_0$ is the maximum element for which $f(\beta_0) > \beta_0$, then $f(\beta_1)<\beta_1$. If we had $f(\beta_1)>\beta_1$, this would contradict the fact that $\beta_0$ is the largest value for which $f(\beta_0) > \beta_0$. 

Under our assumption that there are no cyclic points in $[\beta_0, \beta_1]$ and $B$ is a cycle of period greater than 2, one can see that $f(\beta_1)=\beta_2<\beta_0$.
 
Therefore, 
$f([\beta_0, f(\beta_0)])\supset [\beta_1, \beta_2]$.

Therefore, by the Intermediate Value Theorem, there is a point $c \in [\beta_0, \beta_1]$, such that $f(c)=\beta_0$.

Hence, $\beta_0 \in f([\beta_0, f(\beta_0)])$. By Lemma \ref{lemma1}, there is a point of period 2.

Consider the second case.

For this case, by the assumption on $\beta_0$, one can see that

\begin{eqnarray*}
f(\beta_0) &=& \beta_1>\beta_0 \; \textbf{and} \; \beta_1 \in [\beta_0, \beta_1] \\
f(\beta_1) &=& \beta_2<\beta_1 \; \textbf{and} \; \beta_2 \in [\beta_0, \beta_1] \\
f(\beta_2) &=& \beta_3<\beta_2 \; \textbf{and} \; \beta_3 \in [\beta_0, \beta_1] \\
& \cdots &
\end{eqnarray*}

Since, $B$ is a cycle, in order to make sure that the loop is completed, there must exist a point $\beta_j \in [\beta_0, \beta_1]$, where $j \in \N$, such that $f(\beta_j) \leq \beta_0$.

Observe that both $\beta_1$ and $\beta_j$ are in 
$[\beta_0, \beta_1]$, 
which implies that $f([\beta_0, \beta_1]) \supset [\beta_0, \beta_1]$. Applying the Intermediate Value Theorem, one can see that there exists a point $c$ in  $[\beta_0, \beta_1]$, such that $f(c)=\beta_0$. 

Thus, in both cases we have shown that
there  exists a point $c$ in $[\beta_0, \beta_1]$, such that $f(c)=\beta_0$, 
which is equivalent to $\beta_0 \in f([\beta_0, f(\beta_0)])$. By Lemma \ref{lemma1}, there is a point of period $2$.
\end{proof}

\begin{corollary}
\label{cor1} 
If $f$ has a cycle of period $2^l$, with $l \geq 0$, then $f$ has cycle of periods $2^i$, with $i=0,1,2, \ldots, l-1$.
\end{corollary}

\begin{proof}

Since  $f:I \to I$, then   
by Theorem \ref{fixed1} there exists a fixed point, or a point of a period $1$, or for the sake of our proof, period $2^0$.

If $f$ has a period $2$, so that $l=1$, it is obvious that $f$ has a fixed point due to the same reason stated above, so that $f$ also has cycle of period $1$, or $2^0$.

The next case is for $l=2$, when $f$ has a cycle of period $2^2=4$.
Then,  by Lemma \ref{lemma2}, 
$f$ has a cycle of period $2$,
and by Theorem \ref{fixed1}, $f$ has a cycle of period $1$.

Assume that  $l \geq 3$. 
Let us introduce the new function $g$ defined as $g=f^{2^{i-1}}$, where $i$ is a positive integer such that $i<l$. The case for $i=1$ is governed by previous Lemma \ref{lemma2}.

Observe the following: if $f$ has cycle of period $k \in \N$ and $g=f^h$, where $h \in \N$ is a factor of $k$, then $g$ has a cycle of period $\displaystyle \frac{k}{h}$.

Similarly, knowing the cycle of a function composition, we can go backwards, so that if it is known that $g$ has a cycle of period $\displaystyle \frac{k}{h}$ and $g=f^h$, our conclusion that $f$ has a cycle of period $\displaystyle \frac{k}{h} \, h=k$ holds.

Since $f$ has a cycle of period $2^l$, then $g$ has a cycle of period:

\begin{eqnarray*}
\frac{2^l}{2^{i-1}}=2^{l-i+1}.
\end{eqnarray*}

Note that $l \geq 0$ and $i=1,2,\ldots, l-1$, therefore, $2^{l-i+1} >2$.

By Lemma \ref{lemma2}, $g$ has a cycle of period $2$.

Hence, $f$ also has a cycle of period $\displaystyle 2^{i-1} \cdot 2=2^{i-1+1}=2^i$. 
\end{proof}

\bigskip
\begin{mdframed}

\begin{example}
     \vspace{12pt}
    
    Let $f$ have a cycle of period $32$. By writing $32=2^5$, we know that $l=5$. By Corollary \ref{cor1},  $f$ should also have cycles of periods $2^4, 2^3, 2^2$. Period $2$ is given by Lemma \ref{lemma2}, and period $1$, by the fact that $f: I \to I$ as was shown in section \ref{fixedPts}.

    Let $i=4$, therefore $g=f^{2^3}=f^{8}$. Therefore, $g$ has a cycle of period $\displaystyle \frac{32}{8}=4$. 
Since $4>2$, by Lemma \ref{lemma2}, $g$ must have a cycle of period 2. Therefore, $f$ must have a cycle of period $8 \cdot 2=16$.
    
    By letting $i=2,3$ and repeating the argument shown above, one can see that $f$ indeed has cycles of periods $4,8$.

\end{example}
\end{mdframed}

\vspace{12pt}

\begin{corollary}
\label{cor2}
 If a map $f$ has a cycle of period $\neq 2^i$, with $i=0,1,2,\ldots$, then $f$ has cycles of periods $2^i$, with $i=0,1,2,\ldots$.

In other words, if $f$ has a cycle of period that is not a power of two, 
then $f$ has a cycle of period $2^i$ for every non-negative 
integer $i$.
\end{corollary} 

\begin{proof}
Let us assume that $f$ has a cycle of period $2^l m$, where $l$ is a non-negative integer, and $m$ is an odd number, $m \geq 3$. 
\footnote{As observed earlier, any positive integer that is not a power of two can be written in this way.}

Let us consider two cases.

{\sc Case 1.}  $i$ is a positive integer with $ i \leq l$  (this can only happen if $l \ge 1$). 
As in the proof of Corollary \ref{cor1}, let us define
$g=f^{2^{i-1}}$. 
Then $g$ has a cycle of period
$$ { 2^l m \over 2^{i-1}} = 2^{l-i+1} m , $$
which is bigger than $2$ under our assumptions that $l \geq 1$ and $i$ is a positive integer. 

{\sc Case 2.}  $i$ is a positive integer with $i > l$.
Then $g$ has an orbit of period $m$, which is again bigger than $2$.

Thus, in both cases $g$ has a cycle of period bigger than $2$. Therefore, by 
Lemma \ref{lemma2},

the function $g$ has a $2$-cycle. But then $f$ has a cycle of period
$2 \cdot 2^{i-1} = 2^i$. 
\end{proof}

\bigskip
\begin{mdframed}

\begin{example}
     \vspace{12pt}
    
    Let $f$ have a cycle of period $56$, where $56$ can be uniquely written as $56=2^3 \cdot 7$.  Therefore, $l=3$. Following the idea of the proof, we distinguish 2 cases.
    
    Case 1 consists of $i \leq 3$.
    
    Case 2 governs $i>3$.

    By Lemma \ref{lemma2}, we know that $f$ has a cycle of period $2$ for $i=1$. 
    
    Let $i=2$ (Case 1). Therefore, we use $g=f^{2^{2-1}}=f^2$.
    
    Therefore, $g$ has a cycle of period:
    
    \begin{eqnarray*}
    \frac{2^3 \cdot 7}{2}=2^2 \cdot 7=28>2.
    \end{eqnarray*}
    
    If $g$ has a cycle of period $28$, by Lemma \ref{lemma2}, it has a cycle of period $2$.
    
    Therefore, $f$ has a cycle of period $2^{2-1} \cdot 2=2^2=4$.
    
    Let $i=5$ (Case 2). Let $g=f^{2^{5-1}}=f^{16}$.
    
    We cannot use the argument we used before, since
    
    \begin{eqnarray*}
    \frac{2^3 \cdot 7}{2^4}=\frac{7}{2} \not \in \N.
    \end{eqnarray*}
   
   But $16$ and $7$ are relatively prime, therefore $g$ must have a cycle of period $7>2$. Again, we apply Lemma \ref{lemma2}, so that $g$ has a cycle of period 2, it leads to the conclusion that $f$ has a cycle of period $2^4 \cdot 2=2^5=32$.

   Again, $g$ has a cycle bigger than $2$ and therefore, by Lemma \ref{lemma2} has a cycle of period $2$.
   
  It is convenient to show the existence of $2^i$-orbit in the table below.
    
    \begin{center}
\captionof{table}{Coexistence of cycles of periods $2^i$ for a $56$-cycle.}
\begin{tabular}{| c |c | c | c | c |}
\hline
  $i$        & Case & $g$-function & cycle of $g$ & cycle of $f$   \\ \hline
  $i=1$   & Case 1  & $g=f$  & $2^3 \cdot 7=56$    & $2^0 \cdot 2=2$  \\ \hline
  $i=2$   & Case 1  & $g=f^2$  & $2^2 \cdot 7=28$  &  $2^1 \cdot 2=4$\\ \hline
  $i=3$   & Case 1  & $g=f^4$  & $2^1 \cdot 7=14$  &  $2^2 \cdot 2=8$   \\ \hline
  $i=4$   & Case 2  & $g=f^8$  & $7$   &  $2^3 \cdot 2=16$\\ \hline
  $i=5$   & Case 2  & $g=f^{16}$ & $7$             &  $2^4 \cdot 2=32$ \\  \hline
\end{tabular}
\end{center}
    
    \end{example}
\end{mdframed}

\vspace{12pt}

\vspace{12pt}

\vspace{12pt} 

\section{Sharkovsky Theorem}

Finally, we are ready to discuss the Sharkovsky Theorem \cite{Shark1}. 

\vspace{12pt}

\begin{theorem}
\label{shark}
If $f: I \to I$ is a continuous map that has a cycle of period $m$, then $f$ has cycles of every period $m'$ such that $m \triangleright m'$, where
 \begin{eqnarray*}
& 3 \triangleright 5 \triangleright 7 \triangleright 9 \triangleright  \cdots\\
& 2 \cdot 3 \triangleright 2 \cdot 5 \triangleright 2 \cdot 7 \triangleright 2 \cdot 9 \triangleright \cdots \\
& 2^2 \cdot 3 \triangleright 2^2 \cdot 5 \triangleright 2^2 \cdot 7 \triangleright 2^2 \cdot 9 \triangleright 
                         \cdots\\ 
                         &  \cdots \\
& \cdots \triangleright 2^5 \triangleright 2^4 \triangleright 2^3 \triangleright 2^2 \triangleright 2 
                    \triangleright 1.
\end{eqnarray*}
\end{theorem}

The order of the natural numbers provided by the relation ``$\, \triangleright$'' is called
{\it the Sharkovsky ordering}. 

The heart of the proof of Theorem \ref{shark} will be the 
Lemmas \ref{lemma3} -- \ref{lemma7} considered below.


We begin with some definitions and examples.

For any cycle 
of the function $f$ 
there exists a cyclic permutation, which we denote as $\pi$, a transfer matrix, and an oriented graph, which can be constructed as follows.

Assume that there exists a cycle $B$, that consists of the points $\beta_i$ where $1\leq i \leq m$, none of them equal to each other. Therefore, we can order them so that

\begin{eqnarray*}
\beta_1< \beta_2 < \beta_3 < \cdots< \beta_m.
\end{eqnarray*}

Let the image of a point $\beta_i$ be $f(\beta_i)=\beta_{s_i}$, then the permutation is defined as follows:

\begin{eqnarray*}
\pi = \pi_B = 
\begin{bmatrix}
1 & 2 & 3 & \ldots & m \\
s_1 & s_2 & s_3 & \ldots & s_m
\end{bmatrix}.
\end{eqnarray*}

The set of the next examples in this section will be based on the following cycle of period $4$.

\begin{figure}[ht!]
\centering
\includegraphics[width=90mm]{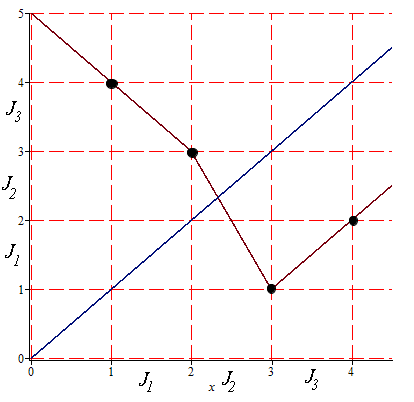}
\caption{Cycle of period 4.}
\label{4cycle}
\end{figure}

The dynamics of the points is shown below

\begin {center}
\begin {tikzpicture}[-latex ,auto ,node distance =2 cm and 3cm ,on grid ,
semithick ,
state/.style ={ circle ,top color =white , bottom color = processblue!20 ,
draw,processblue , text=blue , minimum width =0.4 cm}]
\node[state] (A)
{$1$};
\node[state] (B) [right=of A] {$2$};
\node[state] (C) [right =of B] {$3$};
\node[state] (D) [right =of C] {$4$};
\path (A) edge [bend right = -25] node[right] {} (D);
\path (D) edge [bend left =25] node[below =0.15 cm] {} (B);
\path (B) edge [bend right = -15] node[below =0.15 cm] {} (C);
\path (C) edge [bend left =25] node[above] {} (A);
\end{tikzpicture}
\end{center}

\bigskip
\begin{mdframed}

\begin{example}
\vspace{12pt}

Consider a continuous function with the following properties: 

$f(1)=4, f(2)=3, f(3)=1, f(4)=2$, 

which is shown in Figure \ref{4cycle}. 

We can order the points as follows: $1<2<3<4$, the image of the points will be $4, 3, 1, 2$ respectively.

Therefore, the permutation will be 

\begin{eqnarray*}
\pi =
\begin{bmatrix}
1 & 2 & 3 & 4 \\
4 & 3 & 1 & 2 \\
s_1 & s_2 & s_3 & s_4
\end{bmatrix}.
\end{eqnarray*}

Such representation of a cycle can be very helpful as we will show later.

\end{example}
\end{mdframed}

\bigskip\noindent

Let $J_i$ be the interval 
$J_i=[\beta_i, \beta_{i+1}]$, with $i=1,\ldots, m-1$.
Due to continuity of $f$, the image of $J_i$ satisfies the following:

\begin{displaymath}
   f( J_{\displaystyle i}) \supseteq \left\{
     \begin{array}{lr}
       J_{\displaystyle s_i} \cup \ldots \cup J_{\displaystyle s_{ i+1}-1} & : s_{\displaystyle i} < s_{\displaystyle i+1}\\
       J_{\displaystyle s_{i+1}} \cup \ldots \cup J_{\displaystyle s_{i}-1}& : s_{\displaystyle i} > s_{\displaystyle i+1}
     \end{array}
   \right.
\end{displaymath}

\vspace{12pt}
\begin{mdframed}

\begin{example}
\vspace{12pt}

Continue with the same function (Figure \ref{4cycle}).

\begin{eqnarray*}
\pi =
\begin{bmatrix}
\beta_1 & \beta_2 & \beta_3 & \beta_4 \\
1 & 2 & 3 & 4 \\
4 & 3 & 1 & 2 \\
s_1 & s_2 & s_3 & s_4
\end{bmatrix}.
\end{eqnarray*}

Here we have $m-1 = 3$ intervals, to be precise: $J_1=[1,2], J_2=[2,3], J_3=[3,4]$.

For the interval $J_1$ we have $s_1, s_2$ to be $4,3$ respectively, then $s_1>s_2$.

Therefore, $f(J_1) \supset J_3$.

For the interval $J_2$, from the permutation $\pi$, one can see that $s_2=3, s_3=1$, then $s_2>s_3$.

Hence, $f(J_2) \supset J_1 \cup J_2$.

The image of $J_3$ is different: because, $s_3=1, s_4=2$, then $s_3<s_4$.

Therefore, $f(J_3) \supset J_1$.
\end{example}
\end{mdframed}

As defined in Section \ref{defin}, recall that 
if $f(J_i) \supset J_{s_i}$, we say that $J_i$ covers $J_{s_i}$, 
and write $J_i \to J_{s_i}$.

To any cycle of period $m$ we will associate 
its {\it transition matrix}, which is 
an $n\times n$-matrix, 
where $n = m-1$ is the number of the intervals constituting the cycle.
We define the entries of the transition matrix as $\{\mu_{is}\}$, where

\begin{displaymath}
   \mu_{is} = \left\{
     \begin{array}{lr}
       0 & :  f(J_i) \not \supset J_s,\\
       1 & :  f(J_i) \supset J_s.
     \end{array}
   \right.
\end{displaymath} 

Finally, to a cycle, or a transition matrix, we can also associate an oriented graph,
as shown in the following example.

\vspace{12pt}
\begin{mdframed}

\begin{example}
\vspace{12pt}

Continue with the same function, showed in the previous example.

Here we have $m-1$ intervals, to be precise: $J_1=[1,2], J_2=[2,3], J_3=[3,4]$.

From  Figure \ref{4cycle}, we can see that  

$$f(J_1)=f([1,2]) \supset J_3$$
$$f(J_2)=f([2,3])\supset J_1 \cup J_2$$
$$f(J_3)=f([3,4]) \supset J_1 . $$  

Therefore, a transition matrix is

\begin{eqnarray*}
\pi =
\begin{bmatrix}
0 & 0 & 1 \\
1 & 1 & 0 \\
1 & 0& 0 \\
\end{bmatrix}.
\end{eqnarray*}

The associated graph looks as follows: 

\begin {center}
\begin {tikzpicture}[-latex ,auto ,node distance =2 cm and 3cm ,on grid ,
semithick ,
state/.style ={ circle ,top color =white , bottom color = processblue!20 ,
draw,processblue , text=blue , minimum width =0.4 cm}]
\node[state] (A)
{$J_1$};
\node[state] (B) [right=of A] {$J_3$};
\node[state] (C) [right =of B] {$J_2$};
\path (A) edge [bend right = -15] node[right] {} (B);
\path (B) edge [bend left =15] node[below =0.15 cm] {} (A);
\path (C) edge [bend left = -25] node[below =0.15 cm] {} (A);
\path (C) edge [loop right =25] node[above] {} (C);
\end{tikzpicture}
\end{center}

\end{example}
\end{mdframed}

Given natural numbers $i_1 < i_2$,
the segment of the natural number line contained between $i_1$ and $i_2$, or in set builder notation
\begin{eqnarray*}
\{i \in \N \vert i_1 \leq i \leq i_2\} ,
\end{eqnarray*}
will be denoted as $|i_1, i_2|$.

If $i_2=i_1+1$, then for $|i_1, i_2|$ we will use the special notation $|i_1, *|$.

Given a cyclic permutation $\pi$ of length $n$, let us define the function (operator, map) $A_{\pi}$ as follows. The action of $A_{\pi}$ on a segment $|i_1, i_2| \subset |1,n|$ is:

\begin{eqnarray*}
A_{\pi} |i_1, i_2| =|\underbrace{\min \pi(i)}_{i \in |i_1, i_2|}, \underbrace{\max \pi(i)}_{i \in |i_1, i_2|}|.
\end{eqnarray*}

Further, in order to illustrate some of the statements, we are going to use the following function, which has a cycle of period $8$ (See Figure \ref{cycle8}).

\begin{figure}[ht!]
\centering
\includegraphics[width=80mm]{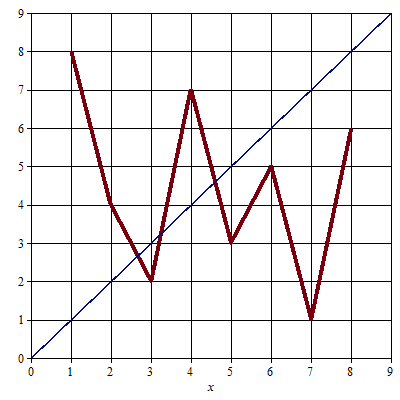}
\caption{Cycle of period 8}
\label{cycle8}
\end{figure}

It is important to understand how the operator $A_{\pi}$ works. Consider the example below.

\vspace{4cm}

\begin{mdframed}

\begin{example}
\vspace{12pt}

Consider the cycle of period $8$, depicted in the Figure \ref{cycle8}, whose permutation $\pi$ is

\begin{eqnarray*}
\pi =
\begin{bmatrix}
1 & 2 & 3 & 4 & 5 & 6 & 7 & 8\\
8 & 4 & 2 & 7 & 3 & 5 & 1 & 6 \\
\end{bmatrix}.
\end{eqnarray*}

Let us pick the interval $|3, 4|$ and investigate the image of its endpoints. 

From Figure  
\ref{cycle8}, one can see that the image of  $3$ is $f(3)=2$, similarly, the image of $4$ is $f(4)=7$ or from the permutation $\pi$:

\begin{center}

\begin{tikzpicture}
\matrix[matrix of math nodes, row sep=4mm] (M) {
 1 & 2 & \color{blue} 3 & \color{blue}4 & 5 & 6 & 7 & 8  \\
 8 & 4 & \color{blue}2 & \color{blue}7 & 3 & 5 & 1 & 6  \\
 };

\draw[my arrow] (M-1-3) to (M-2-3);
\draw[my arrow] (M-1-4) to (M-2-4);
\end{tikzpicture}
\end{center}

We got a set of $2$ elements $\{2,7\}$. Note that there are no numbers other than $2$ and $7$.
Therefore, the minimum is  $\min \{2,7\}=2$ and maximum is $\max\{2,7\}=7$.

Let us pick another interval, $|3,7|$, and study the images of its endpoints. Note that $3$ goes to $2$, while $7$ is sent to $1$.

\begin{center}

\begin{tikzpicture}
\matrix[matrix of math nodes, row sep=4mm] (M) {
 1 & 2 & \color{blue} 3 & 4 & 5 & 6 & \color{blue} 7 & 8  \\
 8 & 4 & \color{blue}2 & \color{blue}7 & \color{blue}3 & \color{blue}5 & \color{blue}1 & 6  \\
 };

\draw[my arrow] (M-1-3) to (M-2-3);
\draw[my arrow] (M-1-7) to (M-2-7);
\end{tikzpicture}
\end{center}

But there are some numbers between $2$ and $1$. Therefore, we get the set consisting of $5$ elements---$\{2,7,3,5,1\}$. 

The minimum is now 
$\min\{2,7,3,5,1\}=1$ and the maximum is $\max\{2,7,3,5,1\}=7$. Hence, the image of $|3,7|$ under $A_{\pi}$ is $|1, 7|$.

\end{example}
\end{mdframed}

In particular, for a map $f: I \to I$, with $n$ cycle  $\beta_1 < \beta_2< \cdots < \beta_n$ and associated permutation $\pi$, the action of $A_{\pi}$ is connected with $f$ as follows:

\begin{eqnarray}
\label{ftopi}
\text{\; if\;\;} A_{\pi}|i_1, i_2|=|i'_1, i'_2| \text{\; then\;\;} f([\beta_{i_1}, \beta_{i_2}]) \supset [\beta_{i'_1}, \beta_{i'_2}].
\end{eqnarray}

Also, 
for any cyclic permutation $\pi$ of length $n$ and any positive integers $i_1 < i_2$, the map
$A_{\pi}$ has the following properties:

\begin{enumerate}
\label{features}
\item $\mathbf{card} A_{\pi} |i_1, i_2| \geq \mathbf{card} |i_1, i_2|$,

\item $A_{\pi}|i_1, i_2| \subset |i_1, i_2|$ if and only if $|i_1, i_2| = |1, n|$,

\item If $|i_1, i_2| \subset |j_1, j_2|$, then $A_{\pi}|i_1, i_2| \subset A_{\pi}|j_1, j_2|$.
\end{enumerate}

Since these properties are essential for the proof of the Theorem, 
we decided to include justifications for all of them.

\begin{proof} 

\begin{itemize}

\item[(1)]

Since a cyclic permutation is a one-to-one map, the cardinality of the set 
$Y = \{ \pi(i) \, \vert \, i_1 \le i \le i_2\} $
is the same as that of $\vert i_1, i_2 \vert$. 
 
Hence, the cardinality of $A_{\pi} |i_1, i_2|  = | \min Y, \max Y|$ 
is at least as big as the cardinality of $\vert i_1, i_2 \vert$.

\bigskip
\bigskip

\item[(2)] 
If $\vert i_1, i_2 \vert = \vert 1, n \vert $ then clearly
$A_\pi \vert 1, n \vert \subset \vert 1, n \vert$. 
Conversely, let us assume that $A_{\pi}|i_1, i_2| \subset |i_1, i_2|$,
or $\vert \min Y,  \max Y \vert  \subset \vert i_1, i_2 \vert$,
where we have again called
$Y = \{ \pi(i) \, \vert \, i_1 \le i \le i_2\} $. 
This means that 
$$ i_1 \leq \min Y < \max Y \leq i_2 . $$
Since we have already proved that the cadinality of $\vert \min Y, \max Y \vert$ is no less than 
that of $\vert i_1, i_2 \vert$, we conclude that both sets have the same cardinality. In
particular, $\min Y = i_1$ and $\max Y = i_2$. Hence, $A_\pi$ sends the segment $\vert i_1, i_2 \vert$
onto itself. Since $A_\pi$ is by assumption a cyclic permutation of length $n$, this can only happen
when the cardinality of $\vert i_1, i_2 \vert$ is $n$, that is, when $\vert i_1, i_2 \vert = \vert 1, n \vert$.

\item[(3)]
Assume that  $|i_1, i_2| \subset |j_1, j_2|$, 
Then
$Y \subset Z$, where 
$$ Y = \{ \pi(i) \, \vert \, i_1 \le i \le i_2\}  \qquad\hbox{ and } \qquad
Z = \{ \pi(j) \, \vert \, j_1 \le j \le j_2\} . $$

Hence, $\min Z \le \min Y < \max Y \le \max Z$,
from which the inclusion 
$A_{\pi}|i_1, i_2| \subset A_{\pi}|j_1, j_2|$ follows.

\end{itemize}
\end{proof}

\vspace{12pt}

It is convenient to adopt the following notation $|i_1, i_2| \rightarrow |i'_1, i'_2|$ meaning $A_{\pi}|i_1, i_2| \supset |i'_1, i'_2|$.

Let $A^k_{\pi}=\underbrace{A_{\pi} \circ \cdots A_{\pi}}_{k-\text{times}}$.

In order to show the proof of the following Lemmata, for a given cyclic permutation $\pi$ we will need a special point, which we will call $i_0$, and is defined as follows:

\begin{equation}
\label{i0}
i_0=\max \left\{ i \in |1, n-1| \big \vert \pi(i)>i \right\}.
\end{equation}

\begin{mdframed}

\begin{example}
\vspace{12pt}

Consider the cycle of period $8$, whose permutation $\pi$ is

\begin{eqnarray*}
\pi =
\begin{bmatrix}
1 & 2 & 3 & 4 & 5 & 6 & 7 & 8\\
8 & 4 & 2 & 7 & 3 & 5 & 1 & 6 \\
\end{bmatrix}.
\end{eqnarray*}

By the definition of $i_0$, we need to find \textbf{the largest} $i$ such that $\pi(i)>i$. 

Observe that  the $i$'s are in the first row, while  the
$\pi(i)$ are in the second row. 

There are three pairs for which $\pi(i)>i$.

\begin{eqnarray*}
\pi =
\begin{bmatrix}
\color{blue} 1 & \color{blue} 2 & 3 & \color{blue} 4 & 5 & 6 & 7 & 8\\
\color{blue} 8 & \color{blue} 4 & 2 & \color{blue} 7 & 3 & 5 & 1 & 6 \\
\end{bmatrix}.
\end{eqnarray*}

Therefore, the set of such $i$'s is $\{1,2,4\}$,
and hence  $i_0=\max\{1,2,4\}=4$.

\end{example}
\end{mdframed}

Let us introduce out next result. 

\vspace{18pt}

\begin{lemma}
\label{lemma3}

Let $\pi$ be a cyclic permutation of length $n$ with $n>2$.

\begin{enumerate}
\item There exists $i_0 \in |1, n-1|$ and a natural number $k$ with $1 \leq k \leq n-2$ such that 

$$|i_0, *| \subset A_{\pi}|i_0, *| \subset A^2_{\pi}|i_0, *| \subset \cdots \subset A^k_{\pi}|i_0, *|=|1,n|.$$

\item For any $i_1 \in |1, n-1|$ with $i_1 \neq i_0$, there exist elements $i_j$ of the set $|1, n-1|$ with 
$j=2,3, \ldots, r$, where $2 \leq r \leq k$, such that the $i_j$, $0 \leq j \leq r$, are pairwise distinct, and

$$|i_0, *| \rightarrow |i_r, *| \rightarrow |i_{r-1}, *| \rightarrow \cdot \rightarrow|i_j, *| \rightarrow 
\cdots \rightarrow |i_1, *|.$$
\end{enumerate}

\end{lemma}

Before we move to the proof of first part of the Lemma \ref{lemma3}, let us introduce an example showing what this part of the Lemma states.

\vspace{12pt}
\begin{mdframed}

\begin{example}

\vspace{12pt}

Given the cycle of period 8

\begin {center}
\begin {tikzpicture}[-latex ,auto ,node distance =1 cm and 1.5 cm ,on grid , 
semithick ,
state/.style ={ circle ,top color =white , bottom color = processblue!20 ,
draw,processblue , text=blue , minimum width =0.4 cm}]
\node[state] (A)
{$1$};
\node[state] (B) [right=of A] {$2$};
\node[state] (C) [right =of B] {$3$};
\node[state] (D) [right =of C] {$4$};
\node[state] (E) [right =of D] {$5$};
\node[state] (F) [right =of E] {$6$};
\node[state] (G) [right =of F] {$7$};
\node[state] (H) [right =of G] {$8$};
\path (A) edge [bend right = -25] node[right] {} (H);
\path (B) edge [bend right =-25] node[right] {} (D);
\path (C) edge [bend left = 25] node[below =0.15 cm] {} (B);
\path (D) edge [bend right =-25] node[right] {} (G);
\path (E) edge [bend left = 25] node[below =0.15 cm] {} (C);
\path (F) edge [bend left = 25] node[below =0.15 cm] {} (E);
\path (G) edge [bend left = 25] node[below =0.15 cm] {} (A);
\path (H) edge [bend left = 25] node[below =0.15 cm] {} (F);
\end{tikzpicture}
\end{center}

The permutation $\pi$ is

\begin{eqnarray*}
\pi =
\begin{bmatrix}
1 & 2 & 3 & 4 & 5 & 6 & 7 & 8\\
8 & 4 & 2 & 7 & 3 & 5 & 1 & 6 \\
\end{bmatrix}.
\end{eqnarray*}

Let $i_0=4$, then $|i_0,*|=|4, 4+1|=|4,5|$. After applying the operator $A_{\pi}$, we get:

\begin{eqnarray*}
|4,5| \xrightarrow {A_{\pi}} |3,7|.
\end{eqnarray*}

After applying $A_{\pi}$ one time, the image does not contain the whole interval, or

\begin{eqnarray*}
|3,7| \neq |1,8|.
\end{eqnarray*}

Therefore, iterations under $A_{\pi}$ can be continued.

Further,

\begin{eqnarray*}
|3, 7| \xrightarrow {A_{\pi}} |1,7| \text{\; or equivalently\;\;} |4, 5| \xrightarrow {A^2_{\pi}} |1,7|.
\end{eqnarray*}

Still $|1,7| \neq |1,8|$.

Apply $A_{\pi}$ to $|1,7|$

\begin{eqnarray*}
|1, 7| \xrightarrow {A_{\pi}} |1,8| \text{\; or equivalently\;\;} |4, 5| \xrightarrow {A^3_{\pi}} |1,8|.
\end{eqnarray*}
The last interval is already $|1, 8| = |1, n|$ in this case. 

We conclude that for $i_0=4$, we have $k=3$.

Moreover,

\begin{eqnarray*}
|4,5| \subset |3,7| \subset |1,7| \subset |1,8|.
\end{eqnarray*}

\end{example}
\end{mdframed}

\begin{proof}
In order to prove the first part of the lemma, we define $i_0$ again by
(\ref{i0}), that is, 

\begin{eqnarray*}
i_0=\max \left\{ i \in |1, n-1| \big \vert \pi(i)>i \right\}.
\end{eqnarray*}

For any cycle of period $n>2$, one can see that $\pi (1) >1$. The inequality $\pi(1)<1$ is not possible because $i \in \N$, and the equality $\pi(1)=1$ will close the cycle to a length of $1$, which is against our assumption $n>2$. 

On the other end of the cycle, we have $\pi(n)<n$. If a cycle has  length $n$, then inequality $\pi(n)>n$ contradicts this assumption. On the other hand $\pi(n)=n$ leads to the existence of a fixed point, which is not a cycle of period $n>2$.

Therefore, $i_0 \in |1, n-1|$.

\vspace{12pt}

Consider the interval $|i_0, *|=|i_0, i_0+1|$. By the way we defined $i_0$, one can see that $\pi(i_0)>i_0$.

On the other hand, $\pi(i_0+1)<i_0+1$. If we had $\pi(i_0+1)>i_0+1$, then $i_0$ would not be the maximum point with this property, contradicting the definition of $i_0$. If $\pi(i_0+1)=i_0+1$, then we have a fixed point and it leads to a contradiction that we have a cycle of length $n>2$.

Note  that if $\pi(i_0+1)= i_0$, then iteration of the second point completes the loop and it implies that the cycle is of period $2$, which contradicts our assumption that $n>2$. Therefore, we are left with one possibility: $\pi(i_0+1)< i_0$, hence, $\pi(i_0+1)< i_0<i_0+1 \leq \pi(i_0)$.

The reasons above lead to the conclusion that

\begin{equation}
\label{covers}
|i_0, *| \subseteq A_{\pi}|i_0, *|.
\end{equation}

Then, by property $3$ of the operator $A_{\pi}$,

\begin{eqnarray*}
A_{\pi}|i_0, *| \subset  A_{\pi} \left( A_{\pi}|i_0, *|\right),
\end{eqnarray*}

whence

\begin{eqnarray*}
|i_0, *| \subset A_{\pi}|i_0, *| \subset  A^2_{\pi} |i_0, *|.
\end{eqnarray*}

In general,

\begin{eqnarray*}
|i_0, *| \subset A_{\pi}|i_0, *| \subset A^2_{\pi}|i_0, *| \subset \cdots \subset A^j_{\pi}|i_0, *| \subset \cdots
\end{eqnarray*}
where $j \in \N$.

Since  $|1, n|$ is a finite subset of $\N$, then there cannot be infinitely many strict inclusions in the above sequence, so at some moment two of these sets must coincide. 

Let $k$ be the positive integer defined as 

\begin{eqnarray*}
k=\min \left\{ j \in \N \, \big \vert \, A^j_{\pi}|i_0, *|=A^{j+1}_{\pi}|i_0, *| \right \}.
\end{eqnarray*}

By the second property of operator $A_{\pi}$, we have that 

\begin{eqnarray*}
A^k_{\pi}|i_0, *|=|1, n|.
\end{eqnarray*}

Note that by the fact that the length of the cycle is $n>2$ (strict inequality), we have 

\begin{eqnarray*}
|i_0, *| \neq |1, n|.
\end{eqnarray*}

This discovery leads to $k \geq 1$. Since, $|i_0, *|=|i_0, i_0+1|$ and all $i$-ths are natural numbers, we have $\mathbf{card}|i_0, *|=2$.

Therefore, 

\begin{eqnarray*}
\mathbf{card} A^j_{\pi}|i_0, *| \geq 2+j,
\end{eqnarray*}
with $0\leq j \leq k$. 
Indeed, 
this follows from the fact that we are  taking $j \leq k$: up to $k$ inclusive, the operator $A_\pi$ increases cardinality, by definition of $k$.

Hence, $k \leq n-2$.

This step completes the proof of the first part of the Lemma.

\vspace{18pt}

Before we proceed to the proof of the second part of this Lemma, let us introduce an example showing how this part of the Lemma works.

\begin{mdframed}

\begin{example}   
\label{exampleTen} 
\vspace{12pt}

We are going to use the same cycle as we had in previous example.

\begin{eqnarray*}
\pi =
\begin{bmatrix}
1 & 2 & 3 & 4 & 5 & 6 & 7 & 8\\
8 & 4 & 2 & 7 & 3 & 5 & 1 & 6 \\
\end{bmatrix}.
\end{eqnarray*}

By the definition, $i_0=4$, therefore, we can pick any $i_1 \in |1, n-1|$ with $i_1 \neq i_0$. Let $i_1=7$.

From the previous example, we have

\begin{eqnarray*}
|4,5| \xrightarrow {A_{\pi}} \underbrace{ |3,7|}_{j=1} \xrightarrow {A_{\pi}} \underbrace{|1,7|}_{j=2} \xrightarrow {A_{\pi}} \underbrace{|1,8|}_{j=3}.
\end{eqnarray*}

If $i_1=7$, we get the interval $|7, *|=|7, 7+1|=|7,8|$.
Note that

\begin{eqnarray*}
|7,8| \not \subset |3,7| \text{ \;\; for\;\;} j=1\\
|7,8| \not \subset |1, 7| \text{ \;\; for\;\;} j=2\\
|7,8| \subset |1, 8| \text{ \;\; for\;\;} j=3\\
\end{eqnarray*}

It means that  the $A_{\pi}$-image of the segment $|1,7|$ contains $|7,8|$, therefore, the $A_{\pi}$-image of some segment of consecutive integers $|1,2|,|2,3|,|3,4|, \ldots ,|6,7|$ must contain $|7,8|$, because otherwise the $A_{\pi}$-image of the entire segment $|1,7|$ could not possibly contain $|7,8|$. 

Let us find the $A_{\pi}$-images of all those segments:

\begin{eqnarray*}
|1,2| \xrightarrow {A_{\pi}} |4,8| \supset |7,8| \\
|2,3| \xrightarrow {A_{\pi}} |2,4| \not \supset |7,8| \\
|3,4| \xrightarrow {A_{\pi}} |2,7| \not \supset |7,8| \\
|4,5| \xrightarrow {A_{\pi}} |3,7| \not \supset |7,8| \\
|5,6| \xrightarrow {A_{\pi}} |3,5| \not \supset |7,8| \\
|6,7| \xrightarrow {A_{\pi}} |1,5| \not \supset |7,8| \\
\end{eqnarray*}

We can see that segment $|1,2|$ works, because the $A_{\pi}|1,2|=|4,8| \supset |7,8|$.

Therefore, we pick $i_2=1$.

Next, looking back at the $A_{\pi}$-sequence starting at $|4,5|$, we know that

\begin{eqnarray*}
|3,7| \xrightarrow {A_{\pi}} |1,7|,
\end{eqnarray*}

and hence

\begin{eqnarray*}
A_{\pi}|3,7| \supset |1,2|.
\end{eqnarray*}

Therefore, we can argue, as was shown above, that the $A_{\pi}$-image of one of the consecutive intervals $|3,4|,|4,5|, \ldots ,|6,7|$ must also contain $|1,2|$.

\begin{eqnarray*}
|3,4| \xrightarrow {A_{\pi}} |2,7| \not \supset |1,2| \\
|4,5| \xrightarrow {A_{\pi}} |3,7| \not \supset |1,2| \\
|5,6| \xrightarrow {A_{\pi}} |3,5| \not \supset |1,2| \\
|6,7| \xrightarrow {A_{\pi}} |1,5| \supset |1,2| \\
\end{eqnarray*}

Now we see that $|6,7|$ is an appropriate segment. Therefore, we pick $i_3=6$.

Finally, by the $A_{\pi}$ sequence, one can see that

\begin{eqnarray*}
|4,5| \xrightarrow {A_{\pi}} |3,7|,
\end{eqnarray*}
therefore,

\begin{eqnarray*}
A_{\pi}|4,5| \supset |6,7|.
\end{eqnarray*}

Indeed,

\begin{eqnarray*}
|4,5| \xrightarrow {A_{\pi}} |3,7| \supset |6,7|. \\
\end{eqnarray*}

Therefore, we pick $i_4=4=i_0$, and we finally obtain

$$ | i_0, * | = |4,5| \xrightarrow {A_{\pi}} |6,7| \xrightarrow {A_{\pi}} |1, 2 | \xrightarrow {A_{\pi}} 
|7, 8| = | i_1, * | . $$

We also could pick from the set $\{1,2,3,5,6,7\}$ any other integer as $i_1$.

As a second example, let us pick $i_1 = 5$. 

Therefore, the interval we are studying is

\begin{eqnarray*}
|5,*|=|5,5+1|=|5,6|.
\end{eqnarray*}

Since, $|1,8| \supset |5,6|$, we move backwards in the sequence of $A_{\pi}$-images of $|4,5|$. 

We have that 
\begin{eqnarray*}
A^2_\pi |4,5| = |1,7| \supset |5,6|, \\
A_\pi |4,5| = |3,7| \supset |5,6|.
\end{eqnarray*}

Therefore, the image of $|4,5|$ already contains $|5,6|$. Hence, $i_2=4=i_0$, and we have
$$  | i_0, * | = |4,5| \xrightarrow {A_{\pi}} |5,6| = | i_1, *| . $$

\end{example}
\end{mdframed}

Note that since $i_0 \neq i_1$ by our assumption and also $|i_1,*| \subset |1,n|$, so that $A^0_{\pi}|i_0,*|=|i_0,*| \not \supset |i_1,*|$, whereas $A^k_{\pi}|i_0,*|=|1,n| \supset |i_1,*|$, then there exists a natural number $j_1$ such that

\begin{eqnarray*}
& A^{j_1-1}_{\pi}|i_0,*| \not \supset |i_1, *| \text{\; and\:\;} \\
& A^{j_1}_{\pi}|i_0,*| \supset |i_1,*|.
\end{eqnarray*}

Since the $A_\pi$-image of the segment $A^{j_1-1}_{\pi}|i_0,*| $ (because $A_\pi \left(A^{j_1-1}_{\pi}|i_1,*|\right)=A^{j_1}_{\pi}|i_1,*|$) contains $|i_1, *|$, 
then at least one sub-segment of consecutive integers of that segment must also cover
$|i_1, *|$ under $A_\pi$. 

Therefore, we can pick an element $i_2$ such that

\begin{eqnarray*}
&  A^{j_2-1}_{\pi}|i_0,*| \not \supset |i_2, *| \text{\; and\:\;} \\
& A^{j_2}_{\pi}|i_0,*| \supset |i_2,*|
\end{eqnarray*}
for some $j_2 < j_1$.

Again, since the $j_2$-th image of $|i_0*|$ covers $|i_2,*|$, then at least one segment of consecutive integers of the image $|i_0, *|$ under $A^{j_2}_{\pi}$ covers $|i_2,*|$.

Hence, we can identify a segment $|i_3,*|$, such that

\begin{eqnarray*}
&  A^{j_3-1}_{\pi}|i_0,*| \not \supset |i_3, *| \text{\; and\:\;} \\
& A^{j_3}_{\pi}|i_0,*| \supset |i_3,*|
\end{eqnarray*}
for $j_3<j_2<j_1$.

Depending on the cycle, the process can go on.

However, by the first part of the Lemma, we have strict inclusions

\begin{eqnarray*}
A^{j+1}_{\pi}|i_0,*| \supset A^j_{\pi}|i_0,*|,
\end{eqnarray*}
for all $j=1,2, \ldots , k-1$.

Hence, at some moment, we have to choose $i_j=i_0$.

Notice that each time we pick a new $i_s$ from the set $A^{j_s+1}_{\pi}|i_0,*| \setminus A^{j_s}_{\pi}|i_0,*|$, or in other words:

\begin{eqnarray*}
& A^{j_s-1}_{\pi}|i_0,*| \not \supset |i_s, *| \text{\; and} \\
& A^{j_s}_{\pi}|i_0, *| \supset |i_s,*|,
\end{eqnarray*} 
so that we pick an element $i_s$, such that the interval $|i_s,*|$ is covered by applying  the $A_{\pi}$-operator $j_s$ times, but not covered by iteration $j_s-1$.  
Therefore, the $i_s$ are pairwise distinct. 

By our construction, the sequence of  $\displaystyle\left\{j_s\right\}^{k-1}_{1}$ is strictly decreasing.

\end{proof}

From all possible cyclic permutations of length $n, n>2$, we need to distinguish the subset of all permutations that exhibit the following property: 

There exists an element $i^* \in |1,n|$ so that both $i^*$ and $\pi(i^*)$  belong to either $|1, i_0|$ or $|i_0+1, n|$. 

The set of the permutations with such property we denote by $\mathfrak{A}$.

\begin{mdframed}

\begin{example}

\vspace{12pt}

In previous examples we had two different cycles, one of period $4$, let us call it $\pi$

\begin{eqnarray*}
\pi =
\begin{bmatrix}
1 & 2 & 3 & 4 \\
4 & 3 & 1 & 2 \\
\end{bmatrix},
\end{eqnarray*}

and a cycle of period $8$, denoted by $\xi$.

\begin{eqnarray*}
\xi =
\begin{bmatrix}
1 & 2 & 3 & 4 & 5 & 6 & 7 & 8\\
8 & 4 & 2 & 7 & 3 & 5 & 1 & 6 \\
\end{bmatrix}.
\end{eqnarray*}

Also, according to the definition, for $\pi$, we get $i_0=2$, and for $\xi$, $i_0=4$.

For simplicity, all elements from the interval $|1, i_0|$ are in blue, while elements from $|i_0+1, n|$ are in red.

Therefore,

\begin{eqnarray*}
\pi =
\begin{bmatrix}
\color{blue}1 & \color{blue}2 & \color{red}3 & \color{red}4 \\
\color{red}4 & \color{red}3 & \color{blue}1 & \color{blue}2 \\
\end{bmatrix}.
\end{eqnarray*}

Hence, there are no elements for which both an element and its image belong to either $|1, i_0|$ or $|i_0+1, n|$. It leads to the conclusion that $\pi \not \in \mathfrak{A}$.

On the other hand,

\begin{eqnarray*}
\xi =
\begin{bmatrix}
\color{blue}1 & \color{blue}2 & \color{blue}3 & \color{blue}4 & \color{red}5 & \color{red}6 & \color{red}7 & \color{red}8\\
\color{red}8 & \color{blue}4 & \color{blue}2 & \color{red}7 & \color{blue}3 & \color{red}5 & \color{blue}1 & \color{red}6 \\
\end{bmatrix}.
\end{eqnarray*}

Here we have such elements, to be exact, $i \in \{2,3,6,8\}$. Therefore, $\xi \in \mathfrak{A}$.

\end{example}
\end{mdframed}

\vspace{18pt}

Further, in addition to $i_0$, we are going to define and use some special numbers ($i_1,i_2, i_3,$ etc.) in order to prove the following Lemma.

\vspace{15pt}

\begin{center}

\begin{tikzpicture}[mydrawstyle/.style={draw=black, thick}, x=1mm, y=1mm, z=1mm]
  \draw[mydrawstyle, -> ](-5,30)--(140,30);
  \path [draw=black, fill=black] (0,30) circle (1.5) node[below=1]{$1$};
  \path [draw=black, fill=red] (70,30) circle (4pt)node[below=1]{$i_0$};
  \path [draw=black, fill=black] (60,30) circle (4pt)node[below=1]{$i_0-1$};
  \path [draw=black, fill=black] (135,30) circle (4pt)node[below=1]{$n$};
  \fill[pattern=north east lines, pattern color=blue] (0,30)--(59,30)--(59,32)--(0,32); \path [draw=black, fill=black] (35,30) circle (0) node[above=1.5]{$1 \leq i_1 \leq i_0-1$};
  
  \draw[mydrawstyle, ->](-5,5)--(140,5);
  \path [draw=black, fill=black] (0,5) circle (1.5) node[below=1]{$1$};
  \path [draw=black, fill=black] (60,5) circle (4pt)node[below=1]{$i_0-1$};
  \path [draw=black, fill=black] (70,5) circle (4pt)node[below=1]{$i_0$};
  \path [draw=black, fill=black] (80,5) circle (4pt)node[below=1]{$i_0+1$};
  \path [draw=black, fill=black] (135,5) circle (4pt)node[below=1]{$n$};
  \fill[pattern=north east lines, pattern color=blue] (1,5)--(59,5)--(59,7)--(1,7);
  \path [draw=black, fill=black] (35,5) circle (0) node[below=1.5]{$1 \leq \pi(i_1) \leq i_0$};
  \path [draw=black, fill=black] (110,5) circle (0) node[below=1.5]{$i_0+1 \leq \pi(i_0) \leq n$};
  \fill[pattern=north east lines, pattern color=red] (81,5)--(134,5)--(134,7)--(81,7);
  
  \draw[->, blue, thick] (20,29)--(35,7) node [midway, sloped, above] {$A_{\pi}$};
  \draw[->, red, thick] (70,30)--(110,7) node [midway, sloped, above] {$A_{\pi}$};

\end{tikzpicture}

\end{center}

\begin{lemma}
\label{lemma4}

 Let $\pi \in \mathfrak{A}$ be of length $n$.

\begin{enumerate}
\item There exist pairwise distinct elements $i_j \in |1, n-1|$, with 
$0 \leq j \leq r, 1 \leq r \leq n-2$, such that the operator $A_{\pi}$ acts as follows:

\begin {center}
\begin {tikzpicture}[-latex ,auto ,node distance =2 cm and 3cm ,on grid ,
thick ,
state/.style ={ circle ,top color =white , bottom color = processblue!20 ,
draw,processblue , text=black , minimum width =0.4 cm}]
\node[] (A)
{$|i_0,*|$};
\node[] (B) [below right=1.6cm and 2.2cm of A] {$|i_r,*|$};
\node[] (C) [below right=4.4cm and 1.2cm of A] {$|i_{r-1},*|$};
\node[] (D) [below left=4.4cm and 1.2cm of A] {$|i_2,*|$};
\node[] (E) [below left=1.6cm and 2.2cm of A] {$|i_1,*|$};
\path [dashed] (C) edge [bend right = 0] node[right]  {} (D);
\path (A) edge [bend left =0] node[below =0.15 cm] {} (B);
\path (B) edge [bend left =0] node[below =0.15 cm] {} (C);
\path (D) edge [bend left =0] node[below =0.15 cm] {} (E);
\path (E) edge [bend left =0] node[below =0.15 cm] {} (A);
\path (A) edge [loop above =45] node[above] {} (A);
\end{tikzpicture}
\end{center}

\item There exist elements $i_1<i_2<i_3$ of the set $|1,n|$ such that the operator $A^2_{\pi}$ acts as follows:

\begin {center}
\begin {tikzpicture}[-latex ,auto ,node distance =2 cm and 3cm ,on grid ,
thick ,
state/.style ={ circle ,top color =white , bottom color = processblue!20 ,
draw,processblue , text=black , minimum width =0.4 cm}]
\node[] (A)
{$|i_1,i_2|$};
\node[] (B) [right=4cm of A] {$|i_2,i_3|$};
\path (A) edge [bend left =35] node[below =0.15 cm] {} (B);
\path (B) edge [bend left =35] node[below =0.15 cm] {} (A);
\path (A) edge [loop left =45] node[left] {} (A);
\path (B) edge [loop right =45] node[right] {} (B);
\end{tikzpicture}
\end{center}

\end{enumerate}
\end{lemma}

\begin{proof}

\begin{enumerate}
\item  In Part 1 of Lemma \ref{lemma3}, 
formula (\ref{covers}), 
we demonstrated that $|i_0,*| \subset A_{\pi}|i_0,*|$, or in other words, the interval $|i_0,*|$ covers itself.

In Part 2 of Lemma \ref{lemma3}, we proved that for any $i_1 \in |1, n-1|, i_0 \neq i_1$, there exists the following chain:
\begin{eqnarray}
|i_0, *| \rightarrow |i_r, *| \rightarrow |i_{r-1}, *| \rightarrow
\cdots \rightarrow|i_j, *| \rightarrow  \cdots  \rightarrow |i_1, *|,
\end{eqnarray}
with $r \leq n-2$.

We need to show that we can pick $i_1$ in such a way that also  $|i_1,*| \rightarrow |i_0,*|$.

Note that by the assumption the cyclic permutation $\pi \in \mathfrak{A}$, therefore by the definition at least one of the following sets are nonempty:

\begin{eqnarray*}
\left\{i \in |1,i_0|  \, \big \vert \, \pi(i) \in |1, i_0|\right\} \neq \emptyset \text{\; Case 1\;},\\
\left\{i \in |i_0+1,n | \, \big \vert \, \pi(i) \in |i_0+1, n|\right\} \neq \emptyset \text{\; Case 2\;}. \\
\end{eqnarray*}

Let us proceed by cases.

\begin{itemize}
\item \textbf{Case 1}. 

Assume that the set 
$\left\{i \in |1,i_0| \, \big \vert \, \pi(i) \in |1, i_0|\right\}$ is not empty. Then we define $i_1$ as follows:

\begin{eqnarray*}
i_1=\max \left\{ i \in |1, i_0| \, \big\vert \, \pi(i) \in |1, i_0| \right\}.
\end{eqnarray*}

Note that the element $i_0$ divides the cyclic permutation into two intervals: $|1, i_0|$ and $|i_0+1, n|$. In this case both $i_1, \pi(i_1) \in |1, i_0|$, while $\pi(i_0) \in |i_0+1, n|$. 

Therefore, 
\begin{eqnarray*}
i_1 < i_0 \text{\; and\;} \pi(i_1) \leq i_0.
\end{eqnarray*}

Keeping in mind that $i \in \N$, we have:

\begin{eqnarray*}
\pi(i_1) < i_0 < i_0+1.
\end{eqnarray*}

Also, by definition of $i_1$, we know that the image of $i_1+1$ no longer belongs to $|1, i_0|$ or, in mathematical notation, $\pi(i_1+1) \in |i_0+1, n|$; therefore,

\begin{eqnarray*}
i_0+1 \leq \pi(i_1+1) \leq n.
\end{eqnarray*}

Having combined all the inequalities, one gets:

\begin{eqnarray*}
&\pi(i_1) < i_0 < i_0+1 \leq \pi(i_1+1), \\
\end{eqnarray*}

Therefore, it leads to the conclusion that the image of $|i_1,*|=|i_1, i_1+1|$ covers $|i_0, *|=|i_0, i_0+1|$.

This completes the proof of Case 1.

\item \textbf{Case 2}.

Assume that
$\left\{i \in |i_0+1,n| \big \vert \pi(i) |i_0+1, n|\right\} \neq \emptyset$. Here we define $i_1$ in a different way:

\begin{eqnarray*}
i_1=\min \left\{ i \in |i_0+1, n| \big \vert \pi(i) \in |i_0+1, n| \right\}-1.
\end{eqnarray*}

Therefore, the image of $i_1$ belongs to $|1, i_0|$, while the image of $|i_1+1|$ is in $|i_0+1, n|$.

Note that $\pi(i_1) \leq i_0$ and $\pi(i_1+1) \geq i_0+1$.

We thus obtain the following inequalities:

\begin{eqnarray*}
\pi(i_1) \leq i_0 < i_0+1 \leq \pi(i_1+1).
\end{eqnarray*}

Hence, the image of  $|i_1,*|$ covers $|i_0,*|$.

\end{itemize}

In addition to the chain 
$|i_0, *| \rightarrow |i_r, *| \rightarrow |i_{r-1}, *| \rightarrow
\cdots \rightarrow|i_j, *| 
\rightarrow \cdots \rightarrow |i_1, *|,$ 
we are guaranteed to have by Lemma \ref{lemma3}, we have showed that $|i_1, *| \rightarrow |i_0, *|$ for our chosen value $i_1$. This completes the proof of the Part 1 of the Lemma \ref{lemma3}.

Note that by the way we defined $i_0$ and as we demonstrated, $| i_0,*|$ covers itself. 

This completes the proof of the first statement of Lemma \ref{lemma4}.

\item 

In order to prove the second Part of this Lemma, we are going to follow the steps below:

\begin{center}
\begin{tikzpicture}[node distance=2cm]
\tikzstyle{startstop} = [rectangle, minimum width=3cm, minimum height=1cm,text centered, draw=black, fill=white!30]
\node (A) [startstop] {Show that $i_1 < i_2 < i_3$}; 
\node (B) [startstop, below of=A] {Proof the existence of $i_4$};
\node (C) [startstop, below of=B] {Complete the proof of Lemma using $i_4$};

\tikzstyle{arrow} = [very thick,->,>=stealth]
\draw [arrow] (A) -- (B);
\draw [arrow] (B) -- (C);

\end{tikzpicture}
\end{center}

\textbf{\color {black} Step 1}.

Without loss of generality  we assume that

\begin{eqnarray*}
\left\{i \in |1,i_0| \, \big \vert \, \pi(i) \in |1, i_0|\right\} \neq \emptyset.
\end{eqnarray*}

Similarly to the previous part we define $i_1$ as

\begin{eqnarray*}
i_1=\max \left\{ i \in |1, i_0| \, \big \vert \, \pi(i) \in |1, i_0| \right\}.
\end{eqnarray*}

As the next step, we define $i_2$ so that $\pi(i_2)$ is the maximum of  the $\pi(i)$ for $i\in |i_1, i_0|${\color{blue},} 
that is

\begin{eqnarray*}
\pi(i_2)=\max \left\{\pi(i)\big \vert i \in |i_1, i_0| \right \}.
\end{eqnarray*}

Finally, let us set 

\begin{eqnarray*}
i_3=i_0+1.
\end{eqnarray*}

Consider the interval $|i_1,i_0|$ and the element $i_2$ in it.

Note that $i_1<i_0$, and the image of $i_1$ by its definition stays in $|1, i_0|$, while the images of the elements 
after $i_1$, including the image of $i_0$, belong to $|i_0+1, n|$. Therefore, we can conclude that

\begin{eqnarray*}
i_1 <i_2 \leq i_0.
\end{eqnarray*}

Also, we clearly have $i_0<i_0+1$. 
Therefore,

\begin{eqnarray*}
&i_1<i_2 \leq i_0 <i_0+1=i_3, \\
\end{eqnarray*}
that is

\begin{eqnarray*}
i_1 < i_2 <i_3.
\end{eqnarray*}

This completes the first step.

\vspace{15pt}

\textbf{\color {black} Step 2}.

Let us call for a moment

$$ J = | i_0 + 1, \pi(i_2)| \qquad \hbox{ and } \qquad
    K = | i_1+1, \pi(i_2) | , $$
so that $ J \subset K$. Let us also label the difference set:
$L = K \setminus J = |i_1 + 1, i_0|$. 

Claim: there exists an element $i_4 \in J$ such that $\pi(i_4) \leq i_1$. 
We will prove this by contradiction. To this end, let us first of all check a couple of facts:
\begin{itemize}

\item $A_\pi L \subset J$. Indeed, if $i \in L$ then $\pi(i) \leq \pi(i_2)$, by definition of
$i_2$. On the other hand, if $i \in L$ then, by definition of $i_1$, $\pi(i)$ no longer lies
in $|1, i_0|$, whence $\pi(i)>i_0$, and therefore $\pi(i) \ge i_0+1$. This shows that
$\pi(i) \in J$ for all such $i$. 

\begin{center}

\begin{tikzpicture}[mydrawstyle/.style={draw=black, thick}, x=1mm, y=1mm, z=1mm]


  \draw[mydrawstyle, -> ](0,30)--(140,30);
  \path [draw=black, fill=black] (10,30) circle (3pt) node[below=1]{$1$};
  \path [draw=black, fill=red] (70,30) circle (3pt)node[below=1]{$i_0$};
  \path [draw=black, fill=black] (80,30) circle (3pt)node[below=1]{$i_0+1$};
  \path [draw=black, fill=black] (120,30) circle (3pt)node[below=1]{$\pi(i_2)$};
  \path [draw=black, fill=black] (40,30) circle (3pt)node[below=1]{$i_1$};
  \path [draw=black, fill=blue] (60,30) circle (3pt)node[below=1]{$i_1+1$};
  \path [draw=black, fill=black] (135,30) circle (3pt)node[below=1]{$n$};
  \fill[fill=blue, opacity=0.5] (61,29)--(69,29)--(69,31)--(61,31); 
  
  
  \draw[mydrawstyle, -> ](0,5)--(140,5) ;
  \path [draw=black, fill=black] (10,5) circle (3pt) node[below=1]{$1$};
  \path [draw=black, fill=black] (70,5) circle (3pt)node[below=1]{$i_0$};
  \path [draw=black, fill=black] (80,5) circle (3pt)node[below=1]{$i_0+1$};
  \path [draw=black, fill=black] (120,5) circle (3pt)node[below=1]{$\pi(i_1+1)$};
  \path [draw=black, fill=black] (40,5) circle (3pt)node[below=1]{$i_1$};
  \path [draw=black, fill=black] (60,5) circle (3pt)node[below=1]{$i_2$};
  \path [draw=black, fill=black] (135,5) circle (3pt)node[below=1]{$n$};
  \fill[pattern=north east lines, pattern color=blue] (80,5)--(119,5)--(119,7)--(80,7); 
  \fill[pattern=north west lines, pattern color=red] (80,5)--(134,5)--(134,8)--(80,8);

  
  \draw[->, blue, thick] (60,30)--(100,7) ;
  \draw[->, red, thick] (70,30)--(110,9) ;

\end{tikzpicture}
\end{center}

\item $\pi(i) < i $ for every $i \in J$. This follows from the definition of $i_0$, plus
the fact that we cannot have $\pi(i) = i$ for any $i$ in a cyclic permutation of order $>1$.
\end{itemize}

\smallskip

Let us now proceed with the proof by contradiction. Assume there is no such $i_4$.
Then we would have $\pi(i) > i_1$ for every $i \in J$. 
We claim that this would then imply that 
$$ A_\pi K \subset K ,  \eqno{(*)} $$
against the fact that $A_\pi$ has no non-trivial invariant subsets of $|1, n|$,
since $\pi$ is a cyclic permutation  
(this is Property 2 of the operator $A_\pi$). 

Let us prove that our contradiction assumption indeed implies (*). 

Since we have $K = J \cup L$, and, 
by the first bullet point above,  $A_\pi L \subset J \subset K$,
then we only need to prove that
$A_\pi J \subset K$.

First, 
if $i \in J$ then $\pi(i)<i$, as pointed out earlier. Therefore,
$\pi(i) < i \leq \pi(i_2)$ (the second inequality follows from the assumption that $i$ lies in $J$). 

Next, if $i \in J$ then, by our contradiction assumption, we have $\pi(i) > i_1$, whence
$\pi(i) \geq i_1 + 1$. 

Both facts together show that indeed $A_\pi J \subset K$. 

Thus, our contradiction assumption indeed implies (*), which is impossible.
Therefore, the existence of such $i_4$ has been established.

The positions of the elements on  the  number line can be expressed by the  inequalities below:

\begin{eqnarray*}
\pi(i_4) \leq i_1 < i_1+1 \leq i_0 <i_0+1 \leq \pi(i_0) < \pi(i_2).
\end{eqnarray*}

\begin{center}

\begin{tikzpicture}[mydrawstyle/.style={draw=black, thick}, x=1mm, y=1mm, z=1mm]
  \draw[mydrawstyle, -> ](-5,30)--(140,30);
  \path [draw=black, fill=black] (0,30) circle (3pt) node[below=1]{$1$};
  \path [draw=black, fill=red] (70,30) circle (3pt)node[below=1]{$i_0$};
  \path [draw=black, fill=black] (80,30) circle (3pt)node[below=1]{$i_0+1$};
  \path [draw=black, fill=black] (100,30) circle (3pt)node[below=1]{$i_4$};
  \path [draw=black, fill=black] (120,30) circle (3pt)node[below=1]{$\pi(i_2)$};
  \path [draw=black, fill=black] (40,30) circle (3pt)node[below=1]{$i_1$};
  \path [draw=black, fill=black] (50,30) circle (3pt)node[below=1]{$i_1+1$};
  \path [draw=black, fill=black] (60,30) circle (3pt)node[below=1]{$i_2$};
  \path [draw=black, fill=black] (135,30) circle (3pt)node[below=1]{$n$};
  \fill[fill=red, opacity=0.5] (51,29)--(69,29)--(69,31)--(51,31); 
  
  
  \draw[mydrawstyle, ->](-5,5)--(140,5);
  \path [draw=black, fill=black] (0,5) circle (3pt) node[below=1]{$1$};
  \path [draw=black, fill=black] (25,5) circle (3pt)node[below=1]{$\pi(i_4)$};
  \path [draw=black, fill=black] (40,5) circle (3pt)node[below=1]{$i_1$};
  \path [draw=black, fill=black] (80,5) circle (3pt)node[below=1]{$i_0+1$};
  \path [draw=black, fill=black] (100,5) circle (3pt)node[below=1]{$\pi(i_0)$};
  \path [draw=black, fill=black] (100,5) circle (0pt)node[above=1]{$ A_\pi|i_1+1, i_0|$};
  \path [draw=black, fill=black] (120,5) circle (3pt)node[below=1]{$\pi(i_2)$};
  \path [draw=black, fill=black] (135,5) circle (3pt)node[below=1]{$n$};
  \fill[pattern=north east lines, pattern color=blue] (1,5)--(39,5)--(39,7)--(1,7);
  \fill[fill=red, opacity=0.5] (81,4)--(119,4)--(119,6)--(81,6);

\end{tikzpicture}

\end{center}

This completes the second step.
\vspace{15pt}

\textbf{\color {black} Step 3}.

Consider the interval $|i_1, i_2|$. As the first step, we need to see the image of this interval under the $A_{\pi}$-operator. 
Since the image of $i_1$ stays in $|1, i_0|$, then it leads to $\pi(i_1) \leq i_0$. On the other hand, $i_4 \in |i_0+1, \pi(i_2)|$, therefore, $i_4 \leq \pi(i_2)$ and $i_0 < i_4$. All together, it becomes:

\begin{eqnarray*}
\pi(i_1) \leq i_0 < i_4 \leq \pi(i_2).
\end{eqnarray*}

Therefore, $|i_1, i_2|$ covers $|i_0, i_4|$.

Now we need to find the image of $|i_0, i_4|$ under the $A_{\pi}$-operator.

By the definitions of the points, the image of $i_0$ belongs to $|i_0+1, n|$ and the point $i_4$ is mapped into $|1, i_1|$. Hence,

\begin{eqnarray*}
\pi(i_4) \leq i_1 < i_2 < i_3 \leq \pi(i_0).
\end{eqnarray*}

Thus, $|i_0, i_4| \rightarrow |i_1, i_3|$ or $|i_1, i_2| \xrightarrow {A^2_{\pi}} |i_1, i_3|$.

\begin{center}

\begin{tikzpicture}[mydrawstyle/.style={draw=black, thick}, x=1mm, y=1mm, z=1mm]


  \draw[mydrawstyle, -> ](0,30)--(140,30) node at (-6,30)[left]{$A^0_{\pi}|i_1, i_2|$};;
  \path [draw=black, fill=black] (10,30) circle (3pt) node[below=1]{$1$};
  \path [draw=black, fill=black] (70,30) circle (3pt)node[below=1]{$i_0$};
  \path [draw=black, fill=black] (80,30) circle (3pt)node[below=1]{$i_3$};
  \path [draw=black, fill=black] (100,30) circle (3pt)node[below=1]{$i_4$};
  \path [draw=black, fill=black] (120,30) circle (3pt)node[below=1]{$\pi(i_2)$};
  \path [draw=black, fill=red] (40,30) circle (3pt)node[below=1]{$i_1$};
  \path [draw=black, fill=orange] (60,30) circle (3pt)node[below=1]{$i_2$};
  \path [draw=black, fill=black] (135,30) circle (3pt)node[below=1]{$n$};
  \fill[fill=blue, opacity=0.5] (41,29)--(59,29)--(59,31)--(41,31); 
  
  
  \draw[mydrawstyle, -> ](0,5)--(140,5) node at (-6,5)[left]{$A^1_{\pi}|i_1, i_2|$};;
  \path [draw=black, fill=black] (10,5) circle (3pt) node[below=1]{$1$};
  \path [draw=black, fill=red] (70,5) circle (3pt)node[below=1]{$i_0$};
  \path [draw=black, fill=black] (80,5) circle (3pt)node[below=1]{$i_3$};
  \path [draw=black, fill=black] (100,5) circle (3pt)node[below=1]{$i_4$};
  \path [draw=black, fill=orange] (120,5) circle (3pt)node[below=1]{$\pi(i_2)$};
  \path [draw=black, fill=white] (40,5) circle (3pt)node[below=1]{$i_1$};
  \path [draw=black, fill=black] (60,5) circle (3pt)node[below=1]{$i_2$};
  \path [draw=black, fill=black] (135,5) circle (3pt)node[below=1]{$n$};
  \fill[fill=blue, opacity=0.5] (71,4)--(119,4)--(119,6)--(71,6); 
  \fill[pattern=north west lines, pattern color=red] (11,5)--(39,5)--(39,7)--(11,7);
  \fill[pattern=north west lines, pattern color=red] (41,5)--(69,5)--(69,7)--(41,7);
  
  
  \draw[mydrawstyle, -> ](0,-20)--(140,-20) node at (-6,-20)[left]{$A^2_{\pi}|i_1, i_2|$};;
  \path [draw=black, fill=black] (10,-20) circle (3pt) node[below=1]{$1$};
  \path [draw=black, fill=black] (70,-20) circle (3pt)node[below=1]{$i_0$};
  \path [draw=black, fill=black] (80,-20) circle (3pt)node[below=1]{$i_3$};
  \path [draw=black, fill=black] (100,-20) circle (3pt)node[below=1]{$i_4$};
  \path [draw=black, fill=black] (120,-20) circle (3pt)node[below=1]{$\pi(i_2)$};
  \path [draw=black, fill=black] (40,-20) circle (3pt)node[below=1]{$i_1$};
  \path [draw=black, fill=black] (60,-20) circle (3pt)node[below=1]{$i_2$};
  \path [draw=black, fill=black] (135,-20) circle (3pt)node[below=1]{$n$};
  \fill[fill=blue, opacity=0.5] (41,-21)--(79,-21)--(79,-19)--(41,-19); 
  \fill[pattern=north west lines, pattern color=red] (81,-20)--(134,-20)--(134,-18)--(81,-18);
  \fill[pattern=north west lines, pattern color=black] (11,-20)--(39,-20)--(39,-18)--(11,-18);

  
  \draw[->, red, thick] (40,29)--(30,7) node [midway, sloped, above] {$A_{\pi}$};
  \draw[->, orange, thick] (60,30)--(118,7) node [midway, sloped, above] {$A_{\pi}$};
  \draw[->, red, thick] (70,5)--(110,-18) node [midway, sloped, above] {$A_{\pi}$};
  \draw[->, black, thick] (100,5)--(20,-18) node [midway, sloped, above] {$A_{\pi}$};

\end{tikzpicture}
\end{center}

Now we need to consider the image of $|i_2, i_3|$. Since $i_1 < i_2$ and $i_1$ was the maximum element whose image stays in $|1, i_0|$, the image of $i_2$ belongs to $|i_0+1, n|$.

Also, $i_0<i_3=i_0+1$ and $i_0$ is the maximum element with $\pi(i) > i$. Therefore, $\pi(i_3) \in |1, i_0|$.

The outcome of the present analysis is:

\begin{eqnarray*}
\pi(i_3) \leq i_0 <i_2 <i_3 \leq i_4 \leq \pi(i_2).
\end{eqnarray*}

Therefore, $|i_2, i_3| \rightarrow |i_0, i_4|$. In a previous step, we demonstrated that $|i_0, i_4| \rightarrow |i_1, i_3|$.

Consequently, $|i_2, i_3| \xrightarrow {A^2_{\pi}} |i_1, i_3|$.

\begin{center}

\begin{tikzpicture}[mydrawstyle/.style={draw=black, thick}, x=1mm, y=1mm, z=1mm]


  \draw[mydrawstyle, -> ](0,30)--(140,30) node at (-6,30)[left]{$A^0_{\pi}|i_2, i_3|$};;
  \path [draw=black, fill=black] (10,30) circle (3pt) node[below=1]{$1$};
  \path [draw=black, fill=black] (70,30) circle (3pt)node[below=1]{$i_0$};
  \path [draw=black, fill=red] (80,30) circle (3pt)node[below=1]{$i_3$};
  \path [draw=black, fill=black] (100,30) circle (3pt)node[below=1]{$i_4$};
  \path [draw=black, fill=black] (120,30) circle (3pt)node[below=1]{$\pi(i_2)$};
  \path [draw=black, fill=black] (40,30) circle (3pt)node[below=1]{$i_1$};
  \path [draw=black, fill=orange] (60,30) circle (3pt)node[below=1]{$i_2$};
  \path [draw=black, fill=black] (135,30) circle (3pt)node[below=1]{$n$};
  \fill[fill=blue, opacity=0.5] (61,29)--(79,29)--(79,31)--(61,31); 
  
  
  \draw[mydrawstyle, -> ](0,5)--(140,5) node at (-6,5)[left]{$A^1_{\pi}|i_2, i_3|$};;
  \path [draw=black, fill=black] (10,5) circle (3pt) node[below=1]{$1$};
  \path [draw=black, fill=black] (70,5) circle (3pt)node[below=1]{$i_0$};
  \path [draw=black, fill=black] (80,5) circle (3pt)node[below=1]{$i_3$};
  \path [draw=black, fill=red] (100,5) circle (3pt)node[below=1]{$i_4$};
  \path [draw=black, fill=orange] (120,5) circle (3pt)node[below=1]{$\pi(i_2)$};
  \path [draw=black, fill=black] (40,5) circle (3pt)node[below=1]{$i_1$};
  \path [draw=black, fill=black] (60,5) circle (3pt)node[below=1]{$i_2$};
  \path [draw=black, fill=black] (135,5) circle (3pt)node[below=1]{$n$};
  \fill[fill=blue, opacity=0.5] (71,4)--(119,4)--(119,6)--(71,6); 
  \fill[pattern=north west lines, pattern color=red] (11,5)--(69,5)--(69,7)--(11,7);
  
  
  \draw[mydrawstyle, -> ](0,-20)--(140,-20) node at (-6,-20)[left]{$A^2_{\pi}|i_2, i_3|$};;
  \path [draw=black, fill=black] (10,-20) circle (3pt) node[below=1]{$1$};
  \path [draw=black, fill=red] (70,-20) circle (3pt)node[below=1]{$i_0$};
  \path [draw=black, fill=black] (80,-20) circle (3pt)node[below=1]{$i_3$};
  \path [draw=black, fill=black] (100,-20) circle (3pt)node[below=1]{$i_4$};
  \path [draw=black, fill=black] (120,-20) circle (3pt)node[below=1]{$\pi(i_2)$};
  \path [draw=black, fill=black] (40,-20) circle (3pt)node[below=1]{$i_1$};
  \path [draw=black, fill=black] (60,-20) circle (3pt)node[below=1]{$i_2$};
  \path [draw=black, fill=black] (135,-20) circle (3pt)node[below=1]{$n$};
  \fill[fill=blue, opacity=0.5] (41,-21)--(79,-21)--(79,-19)--(41,-19); 
  \fill[pattern=north west lines, pattern color=black] (81,-20)--(134,-20)--(134,-18)--(81,-18);
  \fill[pattern=north west lines, pattern color=red] (11,-20)--(39,-20)--(39,-18)--(11,-18);

  
  \draw[->, red, thick] (80,29)--(30,7) node [midway, sloped, above] {$A_{\pi}$};
  \draw[->, orange, thick] (60,30)--(118,7) node [midway, sloped, above] {$A_{\pi}$};
  \draw[->, black, thick] (70,5)--(110,-18) node [midway, sloped, above] {$A_{\pi}$};
  \draw[->, red, thick] (100,5)--(20,-18) node [midway, sloped, above] {$A_{\pi}$};

\end{tikzpicture}
\end{center}

To summarize, we have proved that
$$ |i_1, i_2| \xrightarrow {A^2_{\pi}} |i_1, i_3|   \qquad \hbox{ and } \qquad  
        |i_2, i_3| \xrightarrow {A^2_{\pi}} |i_1, i_3| . $$
 Together with the fact that $i_1 < i_2  < i_3$, this concludes the proof
  of the second part of Lemma \ref{lemma4}.

\end{enumerate}
\end{proof}

\smallskip

An example of application of Lemma \ref{lemma4} is provided below.

\begin{mdframed}
\begin{example}

\vspace{12pt}
 Consider the cyclic permutation $\pi$ of period $8$ (similar to the previous examples).

\begin{eqnarray*}
\pi =
\begin{bmatrix}
1 & 2 & 3 & 4 & 5 & 6 & 7 & 8\\
8 & 4 & 2 & 7 & 3 & 5 & 1 & 6 \\
\end{bmatrix}.
\end{eqnarray*}

We already showed that for $i_1=7$ the following sequence is obtained: $i_2=1, i_3=6, i_4=i_0=4$.

Also, we demonstrated that the sequence of $i$ for $i_1=5$ 
is $i_1=5, i_2=i_0=4$, then we get the following result (depicted below).

\begin {center}
\begin {tikzpicture}[-latex ,auto ,node distance =1.5 cm and 2cm ,on grid ,
thick ,
state/.style ={ circle ,top color =white , bottom color = processblue!20 ,
draw,processblue , text=black , minimum width =0.4 cm}]
\node[] (A)
{$|4,5|$};
\node[] (B) [below right=2cm and 2cm of A] {$|6,7|$};
\node[] (C) [below =4cm of A] {$|1,2|$};
\node[] (D) [below left=2cm and 2cm of A] {$|7,8|$};
\path (A) edge [bend left =0] node[below =0.15 cm] {} (B);
\path (B) edge [bend left =0] node[below =0.15 cm] {} (C);
\path (C) edge [bend left =0] node[below =0.15 cm] {} (D);
\path (D) edge [bend left =0] node[below =0.15 cm] {} (A);
\path (A) edge [loop above =45] node[above] {} (A);
\end{tikzpicture}
\qquad
\begin {tikzpicture}[-latex ,auto ,node distance =1.5 cm and 2cm ,on grid ,
thick ,
state/.style ={ circle ,top color =white , bottom color = processblue!20 ,
draw,processblue , text=black , minimum width =0.4 cm}]
\node[] (A)
{$|4,5|$};
\node[] (C) [right =4cm of A] {$|5,6|$};
\path (A) edge [bend left =30] node[below =0.15 cm] {} (C);
\path (C) edge [bend left =30] node[below =0.15 cm] {} (A);
\path (A) edge [loop above =45] node[above] {} (A);
\end{tikzpicture}
\end{center}

Among other possible choices we could pick $i_1=2$;
then we need to consider the interval $|i_1, *|=|2,3|$.

Arguing as in Example \ref{exampleTen},
we can create a path from $|i_0, *|=|4,5|$ to $|i_1, *|=|2,3|$. 
Namely, $A_\pi |4,5| = |3, 7| \supset |2,3|$, and we notice that we can pick $i_2 = 3$ 
and $i_3 = 4 = i_0$, to get
$$ | i_0, * | = |4,5| \xrightarrow {A_{\pi}} |3,4| \xrightarrow {A_{\pi}} |2,3 |  = | i_1, * | . $$
However, $|2,3|$ does not cover $|4,5|$: 

\begin{center}
\begin {tikzpicture}[-latex ,auto ,node distance =1.5 cm and 2cm ,on grid ,
thick ,
state/.style ={ circle ,top color =white , bottom color = processblue!20 ,
draw,processblue , text=black , minimum width =0.4 cm}]
\node[] (A)
{$|4,5|$};
\node[] (B) [below right=2cm and 2cm of A] {$|3,4|$};
\node[] (C) [below left=2cm and 2cm of A] {$|2,3|$};
\path (A) edge [bend left =0] node[below =0.15 cm] {} (B);
\path (B) edge [bend left =0] node[below =0.15 cm] {} (C);
\path (C) edge [bend left =0] node[sloped, below =-0.3 cm] {\color{red}\textbf{X}} (A);
\path (A) edge [loop above =45] node[above] {} (A);
\end{tikzpicture}
\end{center}

Therefore, we can conclude that picking $i_1=2$ was a bad choice, 
since in this case we cannot close the loop from $i_1$ back to $i_0$  

Lemma \ref{lemma4}, on the other hand, guarantees that we can always make a right choice of $i_1$, 
which would allow us to ``close the loop.'' 

Let us show how we can do this. 

We consider the same permutation again (note that $i_0$ is highlighted). Recall that in the previous example we showed that $\pi \in \mathfrak{A}$.

\begin{eqnarray*}
\pi =
\begin{bmatrix}
1 & 2 &  \color{red}3 & \cellcolor[gray]{.8} 4 & 5 & \color{blue}6 & 7 & \color{blue} 8\\
 8 & 4 &  \color{red} 2 & \cellcolor[gray]{.8} 7 & 3 & \color{blue} 5 & 1 & \color{blue} 6 \\
\end{bmatrix}.
\end{eqnarray*}

In the proof of Lemma \ref{lemma4} we considered two cases.

\begin{eqnarray*}
\left\{i \in |1,i_0|  \, \big \vert \, \pi(i) \in |1, i_0|\right\} \neq \emptyset \text{\; {\color{red}Case 1}\;},\\
\left\{i \in |i_0+1,n | \, \big \vert \, \pi(i) \in |i_0+1, n|\right\} \neq \emptyset \text{\; {\color{blue}Case 2}\;}. \\
\end{eqnarray*}

Luckily, $\pi$ satisfies both cases.

Following the choice prescribed under Case 1, we must pick 
\begin{eqnarray*}
i_1=\max \left\{ i \in |1, i_0| \, \big\vert \, \pi(i) \in |1, i_0| \right\} = 3.
\end{eqnarray*}

Then $A_\pi |3,4| = |2,7| \supset |4,5|$, so indeed the loop is closed. 
Working ``backwards'' for the images of $|4,5|$, as we did in Example \ref{exampleTen}:

\begin{eqnarray*}
|4,5| \xrightarrow {A_{\pi}} \underbrace{ |3,7|}_{j=1} \xrightarrow {A_{\pi}} \underbrace{|1,7|}_{j=2} \xrightarrow {A_{\pi}} \underbrace{|1,8|}_{j=3}.
\end{eqnarray*}
we see that actually the first step works: $|4,5|  \xrightarrow {A_{\pi}} |3,7| \supset |3,4|$ 
(this is of course our luck, it doesn't have to be just a two-loop), 
so we get: 

\begin {center}
\qquad
\begin {tikzpicture}[-latex ,auto ,node distance =1.5 cm and 2cm ,on grid ,
thick ,
state/.style ={ circle ,top color =white , bottom color = processblue!20 ,
draw,processblue , text=black , minimum width =0.4 cm}]
\node[] (A)
{$|4,5|$};
\node[] (C) [right =4cm of A] {$|3,4|$};
\path (A) edge [bend left =30] node[below =0.15 cm] {} (C);
\path (C) edge [bend left =30] node[below =0.15 cm] {} (A);
\path (A) edge [loop above =45] node[above] {} (A);
\end{tikzpicture}
\end{center}

Under Case 2, we should pick
\begin{eqnarray*}
i_1=\min \left\{ i \in |i_0+1, n| \big \vert \pi(i) \in |i_0+1, n| \right\}-1 = 5.
\end{eqnarray*}

Then, again, the theory works, and the loop is closed, since
$A_\pi |5,6| = |3,5| \supset |4,5|$. 
This, also, happens to be our second choice of $i_1$ in 

Example \ref{exampleTen}. 

Now, working again backwards from the $A_\pi$-images of $|4,5|$, we see we are again lucky
and the first iteration already works 
so we indeed have
\begin {center}
\qquad
\begin {tikzpicture}[-latex ,auto ,node distance =1.5 cm and 2cm ,on grid ,
thick ,
state/.style ={ circle ,top color =white , bottom color = processblue!20 ,
draw,processblue , text=black , minimum width =0.4 cm}]
\node[] (A)
{$|4,5|$};
\node[] (C) [right =4cm of A] {$|5,6|$};
\path (A) edge [bend left =30] node[below =0.15 cm] {} (C);
\path (C) edge [bend left =30] node[below =0.15 cm] {} (A);
\path (A) edge [loop above =45] node[above] {} (A);
\end{tikzpicture}
\end{center}

As we demonstrated,  
besides the choice(s) based on Lemma \ref{lemma4}, there might be other $i_1$'s that will work (as long as a permutation belongs to $\mathfrak{A}$). However, the permutation below is an example that demonstrates the existence of permutations with unique $i_1$, 
which is highlighted below. 

\begin{eqnarray*}
\pi =
\begin{bmatrix}
& {\cellcolor[gray]{.8}1} & 2 & 3 & 4 & 5 & 6 & 7 & 8 & 9\\
& {\cellcolor[gray]{.8}4} & 6 & 9 & 8 & 7 & 5 & 3 & 2 & 1 \\
\end{bmatrix}.
\end{eqnarray*}

\end{example}
\end{mdframed}

\bigskip

\begin{lemma}
\label{sublemma_6}
Every cyclic permutation of odd period $n, n \in \N${\color{blue},} 
belongs to $\mathfrak{A}$.
\end{lemma}

\begin{proof}

Proceed by contradiction. We assume that there exists a cyclic permutation of odd period $n$ greater than $1$ that is not in $\mathfrak{A}$:

\begin{eqnarray*}
\begin{bmatrix}
1 & 2 &  \ldots  & i_0 & \ldots & n \\
a_1 & a_2 & \ldots & a_i & \ldots & a_n  \\
\end{bmatrix}.
\end{eqnarray*}

Since $i_0$ cannot be the last element of the cycle, we get two subintervals:

\begin{eqnarray*}
& |1, i_0|,\; \textbf{card}|1, i_0|=i_0; \\
& |i_0+1, n|,\; \textbf{card}|i_0+1, n|=n-i_0. \\
\end{eqnarray*}
Since the permutation is not in $\mathfrak A$,
then the 
images of all the elements from $|i_0+1, n|$ are in $|1, i_0|$, 
and vice-versa. Moreover,  
since any permutation
$\pi$ is a one-to-one map of $|1, n|$ onto itself,  then 
these two sets must have the same cardinality, that is, 

\begin{eqnarray*}
& \mathbf{card}\left(|i_0+1, n|\right)=\mathbf{card}\left(|1, i_0| \right) .
\end{eqnarray*}
Hence, $n- i_0 = i_0$, whence $n = 2 i_0$ is an even number,
against our assumption.
\end{proof}

\bigskip

\begin{lemma}
\label{lemma6}

If a map has a cycle of odd period $n, n>1$, then the map has cycles of any odd period greater than $n$, and cycles of any even period.

\end{lemma}

\begin{proof}
First of all, 
note that, by Lemma \ref{sublemma_6}, a cycle of an odd period $n,n>1$ is a $\mathfrak{A}$-permutation. 

Let us assume that the continuous function $f$ has a cycle of odd period $n$, 
\begin{equation}
\label{cycleB}
B =  \{\beta_1, \beta_2, \ldots, \beta_n\} , 
\end{equation}
with the $\beta_i$ ordered in increasing sequence,
$\beta_1 < \beta_2 < \cdots < \beta_n$. Let us denote 
$f(\beta_i) = \beta_{s_i}$, for $i = 1, \ldots, n$,
and let us consider the associated cyclic $n$-permutation 
\begin{eqnarray*}
\pi =
\begin{bmatrix}
1 & 2 & 3 & \ldots & m \\
s_1 & s_2 & s_3 & \ldots & s_m
\end{bmatrix}.
\end{eqnarray*}
Earlier in this section, we defined the 
related operator $A_\pi$, acting on segments $| i_1, i_2| \subset |1, n|$, 
and established an important connection  between $A_\pi$ and the
given $n$-cycle (\ref{cycleB}). 
This connection was shown in equation  (\ref{ftopi}), which states:
\begin{eqnarray*}
\text{\; if\;\;} A_{\pi}|i_1, i_2|=|i'_1, i'_2| \text{\; then\;\;} f([\beta_{i_1}, \beta_{i_2}]) \supset [\beta_{i'_1}, \beta_{i'_2}].
\end{eqnarray*}
As shown in Lemma \ref{lemma4},
the permutation $\pi$ has the following associated graph,
with $r+1$ vertices, where $r \leq n-2$:
\begin {center}
\begin {tikzpicture}[-latex ,auto ,node distance =2 cm and 3cm ,on grid ,
thick ,
state/.style ={ circle ,top color =white , bottom color = processblue!20 ,
draw,processblue , text=black , minimum width =0.4 cm}]
\node[] (A)
{$|i_0,*|$};
\node[] (B) [below right=1.6cm and 2.2cm of A] {$|i_r,*|$};
\node[] (C) [below right=4.4cm and 1.2cm of A] {$|i_{r-1},*|$};
\node[] (D) [below left=4.4cm and 1.2cm of A] {$|i_2,*|$};
\node[] (E) [below left=1.6cm and 2.2cm of A] {$|i_1,*|$};
\path [dashed] (C) edge [bend right = 0] node[right]  {} (D);
\path (A) edge [bend left =0] node[below =0.15 cm] {} (B);
\path (B) edge [bend left =0] node[below =0.15 cm] {} (C);
\path (D) edge [bend left =0] node[below =0.15 cm] {} (E);
\path (E) edge [bend left =0] node[below =0.15 cm] {} (A);
\path (A) edge [loop above =45] node[above] {} (A);
\end{tikzpicture}
\end{center}

If we denote $J_i = [\beta_i, \beta_{i+1}]$, $i = 1, \ldots, n-1$, then the above graph 
translates into the following graph for the intervals $J_i$, and the function $f$:
\begin {center}
\begin {tikzpicture}[-latex ,auto ,node distance =2 cm and 3cm ,on grid ,
thick ,
state/.style ={ circle ,top color =white , bottom color = processblue!20 ,
draw,processblue , text=black , minimum width =0.4 cm}]
\node[] (A)
{$J_{i_0}$};
\node[] (B) [below right=1.6cm and 2.2cm of A] {$J_{i_r}$};
\node[] (C) [below right=4.4cm and 1.2cm of A] {$J_{i_{r-1}}$};
\node[] (D) [below left=4.4cm and 1.2cm of A] {$J_{i_2}$};
\node[] (E) [below left=1.6cm and 2.2cm of A] {$J_{i_1}$};
\path [dashed] (C) edge [bend right = 0] node[right]  {} (D);
\path (A) edge [bend left =0] node[below =0.15 cm] {} (B);
\path (B) edge [bend left =0] node[below =0.15 cm] {} (C);
\path (D) edge [bend left =0] node[below =0.15 cm] {} (E);
\path (E) edge [bend left =0] node[below =0.15 cm] {} (A);
\path (A) edge [loop above =45] node[above] {} (A);
\end{tikzpicture}
\end{center}
Let us fix an arbitrary integer $m > n$, 
and let us call $k = m - r$.
Since $m > n > r$, and since both $m$ and $n$ are odd, 
then $k$ is a positive integer $\ge 2$.

Consider the following loop of length $m$:

\begin{equation}
\label{loop}
\underbrace{J_{i_0} \to J_{i_0} \to \cdots \to J_{i_0}}_{k \text{ copies}} \to 
     J_{i_r} \to J_{i_{r-1}} \to  \cdots \to J_{i_1}\to J_{i_0}.
\end{equation}

By Lemma \ref{Itin1} (the Itinerary Lemma), there exists a point 
$\gamma \in J_{i_0}$ that follows this loop, such that $f^m(\gamma) = \gamma$. 

In the eventual case that $n$ divides $m$, we need to make sure that $\gamma$ has indeed
least period $m$. 
But this is guaranteed by the second
part of Lemma \ref{Itin1}. 
Indeed, in the first place, the interior of $J_{i_0}$ is disjoint from all other $J_i$ with 
$i \neq i_0$. 
In the second place, 
the endpoints of all the intervals $J_i$ are part of the orbit $B$,
which does not follow the loop (\ref{loop}). 
Indeed, since $n$ divides $m$, then $m \ge n + 2$, which implies that 
$k \geq 3$. But then, if we call $J_{i_0} = [p, q]$, the first three intervals in this loop
would imply that $p$ can only go to $p$ or $q$, and the same for $q$. But then $B$ would
be either a 1-cycle, or a 2-cycle, against the assumption. 

Thus, $\gamma$ has least period $m$.

We have so far proved that
if $f$ has a cycle of odd period $n > 1$, 
then there are cycles of any period greater $n$, including odd and even integers. 

It remains to show that $f$ has cycles of even periods less than $n$.

First of all, $f$ has a $2$-cycle, by Lemma \ref{lemma2}. 
For even numbers greater than $2$, 
we are going to use the second part of Lemma \ref{lemma4},
which states the existence of integers $i_1, i_2, i_3$, with
$1 \leq i_1 < i_2 < i_3 \leq n$, such that the following graph holds for the operator 
$A^2_\pi$: 

\begin {center}
\begin {tikzpicture}[-latex ,auto ,node distance =2 cm and 3cm ,on grid ,
thick ,
state/.style ={ circle ,top color =white , bottom color = processblue!20 ,
draw,processblue , text=black , minimum width =0.4 cm}]
\node[] (A)
{$|i_1,i_2|$};
\node[] (B) [right=4cm of A] {$|i_2,i_3|$};
\path (A) edge [bend left =35] node[below =0.15 cm] {} (B);
\path (B) edge [bend left =35] node[below =0.15 cm] {} (A);
\path (A) edge [loop left =45] node[left] {} (A);
\path (B) edge [loop right =45] node[right] {} (B);
\end{tikzpicture}
\end{center}

Let us denote $I_1 = [\beta_{i_1}, \beta_{i_2}]$, and 
$I_2 = [\beta_{i_2}, \beta_{i_3}]$. Then the above graph translates into the following one,
for the function $f^2$, and the intervals $I_1$, $I_2$:
\begin {center}
\begin {tikzpicture}[-latex ,auto ,node distance =2 cm and 3cm ,on grid ,
thick ,
state/.style ={ circle ,top color =white , bottom color = processblue!20 ,
draw,processblue , text=black , minimum width =0.4 cm}]
\node[] (A)
{$I_1$};
\node[] (B) [right=4cm of A] {$I_2$};
\path (A) edge [bend left =35] node[below =0.15 cm] {} (B);
\path (B) edge [bend left =35] node[below =0.15 cm] {} (A);
\path (A) edge [loop left =45] node[left] {} (A);
\path (B) edge [loop right =45] node[right] {} (B);
\end{tikzpicture}
\end{center}

If we compare this graph with the one in the proof of Theorem \ref{chaos}, 
we notice that it just has an extra arrow. 
Moreover, since $\beta_{i_1} < \beta_{i_2} < \beta_{i_2}$, the interior of $I_1$
is disjoint from the interval $I_2$.

We now can argue as in the proof of Theorem \ref{chaos} and show the existence of
periodic orbits choosing convenient loops. 

Namely, $f^2$ has a fixed point, by 
Theorem \ref{fixed2} applied to the graph $I_1 \to I_1$, and it has an orbit of period
two, by  Lemma \ref{lemma2}. 
Further, if $2k$ is any even number less than $n$, with $k > 1$, then, applying the Itinerary Lemma 
\ref{Itin1} to the loop
\begin{equation*}
I_1 \to \underbrace{I_2 \to I_2 \to I_2 \to 
\cdots  \to I_2}_{k-1 \text{ copies }} \to I_1,
\end{equation*}
we obtain a point $p \in I_1$ that follows the loop and satisfies $(f^2)^k (p) = p$.
Since the endpoints of $I_1$ and $I_2$ belong to the orbit $B$ of period $n > 2k$, 
then $p$ must be interior to $I_1$, and since the interior of $I_1$ is disjoing from
$I_2$, we conclude that $p$ has least period $k$ for $f^2$, and therefore it has
least period $2k$ for $f$.

\end{proof}


\bigskip

\begin{lemma}
\label{lemma7}

If a map has a cycle of period $2^l(2k+1), k \in \N$, then the map has cycles of periods $2^l(2r+1)$ and $2^{l+1}s$, where $r>k, s \in \N$.
\end{lemma}

\begin{proof}

Let $f$ be a map with a cyclic point of period $2^l(2k+1)$, where $k \geq 1$, and 
define the map $g$ as $g=f^{2^l}$. Then $g$ has a cycle of period
\begin{eqnarray*}
\frac{2^l(2k+1)}{2^l}=2k+1.
\end{eqnarray*}
Note that $2k+1$ is an odd number; 
therefore, according to Lemma \ref{lemma6}, $g$ has cycles of all odd periods greater than $2k+1$. Or, in other words, we can say that $g$ has cycles of periods $2r+1$, where $r \in \N, r>k$. If so, then $f$ has cycles of periods
\begin{eqnarray*}
2^l(2r+1).
\end{eqnarray*}

Also, Lemma \ref{lemma6} shows that $g$ has cycles of any even periods $2s$, with $s \geq 1$. Then, $f$ has cycles of periods
\begin{eqnarray*}
2^l \cdot 2s &= & 2^{l+1}s.
\end{eqnarray*}

This step completes the proof of the Lemma.%
\end{proof}


{\bf Proof of the Sharkovsky theorem.}
If $n$ is an odd number greater than $1$, then Lemma \ref{lemma6} guarantees the first
line in the Sharkovsky ordering, 
\begin{eqnarray*}
3 \triangleright 5 \triangleright 7 \triangleright 9 \triangleright  \cdots 2k+1 \cdots 
\end{eqnarray*}
and, moreover, that any number on the first line will also imply all the numbers
in the rows below.

Next, Lemma \ref{lemma7} implies all subsequent rows, except the last one  ($l \in \N)$:

\begin{eqnarray*}
& 2 \cdot 3 \triangleright 2 \cdot 5 \triangleright 2 \cdot 7 \triangleright 2 \cdot 9 \triangleright \cdots \\
& 2^2 \cdot 3 \triangleright 2^2 \cdot 5 \triangleright 2^2 \cdot 7 \triangleright 2^2 \cdot 9 
           \triangleright \cdots \\
& 2^3 \cdot 3 \triangleright 2^3 \cdot 5 \triangleright 2^3 \cdot 7 \triangleright 2^3 \cdot 9 
                 \triangleright \cdots \\
& \cdots \\
& 2^l \cdot 3 \triangleright 2^l \cdot 5 \triangleright 2^l \cdot 7 \triangleright 2^l \cdot 9 \triangleright \cdots \\
&  \cdots \\
\end{eqnarray*}

Corollary \ref{cor2}, plus Theorem \ref{fixed2},  
guarantee that any number in any row of the Sharkovsky ordering,
except the last, implies the last line (the powers of two, including $2^0 = 1$).

Finally, Corollary \ref{cor1}, plus Theorem \ref{fixed2}, guarantee the last row in this ordering:
\begin{eqnarray*}
\cdots \triangleright 2^l \triangleright \cdots \triangleright  2^5 \triangleright 2^4 
\triangleright 2^3 \triangleright 2^2 \triangleright 2 \triangleright 1.
\end{eqnarray*}
This completes the proof of the theorem. \qed

\bibliographystyle{plain}

\end{document}

%% file: macruss.tex
\input cyracc.def
\newfam\cyrfam

\font\twelvecyr=wncyr10 at 12pt
\def\cyrlar{\twelvecyr\cyracc}  

\def\ya/{\'\char31}
\def\Ya/{\'\char23}
\def\yu/{\'\char24}
\def\Yu/{\'\char16}